%% file: TKMA.tex
\documentclass[review,hidelinks,onefignum,onetabnum]{siamart250211}

\input{siims_shared}

\ifpdf
\hypersetup{
  pdftitle={An Example Article},
  pdfauthor={D. Doe, P. T. Frank, and J. E. Smith}
  frenchlinks={false}
}
\fi

\begin{document}
\begin{sloppypar}

\nolinenumbers

\maketitle

\begin{abstract}
In this paper, we propose a Two-step Krasnosel'ski\u{\i}-Mann (KM) Algorithm (TKMA) with adaptive momentum for solving convex optimization problems arising in image processing. Such optimization problems can often be reformulated as fixed-point problems for certain operators, which are then solved using iterative methods based on the same operator, including the KM iteration, to ultimately obtain the solution to the original optimization problem. Prior to developing the TKMA, we introduce a KM iteration enhanced with adaptive momentum, derived from geometric properties of an $\alpha$-averaged nonexpansive operator $T$ with $\alpha\in(0,1)$, KM acceleration technique, and information from the composite operator $T^2$. The proposed TKMA is constructed as a convex combination of this adaptive-momentum KM iteration and the Picard iteration of $T^2$. We prove that the sequence generated by TKMA converges weakly to a fixed point of $T$ in a real Hilbert space. Moreover, under $\alpha\in(0,1/2]$ and specific assumptions on the adaptive momentum parameters, we prove that the algorithm achieves an $o(1/\sqrt{k})$ convergence rate in terms of the distance between successive iterates. Numerical experiments demonstrate that TKMA outperforms the FPPA, PGA, Fast KM algorithm, and Halpern algorithm on tasks such as image denoising and low-rank matrix completion.
\end{abstract}
\begin{keywords}
Two-step Krasnosel'ski\u{\i}-Mann algorithm, adaptive momentum, convex optimization, image denoising, matrix completion
\end{keywords}
%
\begin{AMS}
  49M37, 65K05, 90C25
\end{AMS}

\renewcommand{\thefootnote}{}
\footnotetext{$^*$Corresponding author.}

\section{Introduction}\label{sec:intro}
Fixed-point type algorithms have been popular in solving nondifferentiable convex or nonconvex optimization problems such as image processing \cite{chen2013primal,li2015multi,lu2016multiplicative,micchelli2011proximity,micchelli2013proximity,shen2016wavelet}, medical imaging \cite{krol2012preconditioned,lin2025accelerated,lin2019krasnoselskii,ross2017relaxed,zheng2019sparsity}, machine learning \cite{chen2015convergence,li2018fixed,li2019two,polson2015proximal}, and compressed sensing \cite{figueiredo2008gradient,zhu2015fast}. This popularity largely arises from their effectiveness in efficiently addressing optimization problems that involve non-differentiable terms, which are common in practical engineering scenarios. Additionally, their typically straightforward iterative structures make them appealing for implementation. Let $\mH$ be a real Hilbert space with inner product $\langle\cdot,\cdot\rangle$ and induced norm $\|\cdot\|$. In this paper, we are interested in formulating a fast fixed-point type algorithm for solving the fixed-point problem
$$
\text{Find} \ \ \bmx\in\mH \ \ \text{such that} \ \ \bmx=T(\bmx),
$$
where $T:\mH\rightarrow \mH$ is an operator. A basic and widely used procedure for approximating a fixed point of $T$ is the following process, also called Picard iteration,
\begin{equation}\label{Picarditer}
\bmx^{k+1}=T\bmx^k, \ \ k\in\bbN,
\end{equation}
where $\bmx^0\in\mH$ is a starting point. The Banach fixed-point theorem establishes that the sequence $\{\bmx^k\}$ generated by \eqref{Picarditer} converges strongly to the unique fixed point of $T$ with linear convergence rate, provided that $T$ is contractive. This conclusion does not hold if
$T$ is merely nonexpansive. To relax the restrictive contractiveness requirement, Krasnosel'ski\u{\i} proposed in \cite{krasnosel1955two} to apply the Picard iteration \eqref{Picarditer} not directly to $T$, but instead to the averaged operator $\frac{1}{2}\mI+\frac{1}{2}T$, where $\mI$ denotes the identity operator on $\mH$. Further, the Krasnosel'ski\u{\i}-Mann (KM) iteration is defined via a convex combination of the operators $\mI$ and $T$:
\begin{equation}\label{TradKM}
\bmx^{k+1}=(1-\lambda_k)\bmx^k+\lambda_k T\bmx^k, \ \ k\in\bbN,
\end{equation}
where $\lambda_k\in(0,1]$ and $T$ is nonexpansive. Weak convergence of this iteration was first established for a constant parameter $\lambda_k\equiv\lambda\in(0,1)$ in \cite{schaefer1957uber}, and later extended in \cite{groetsch1972note} to variable parameters under the condition $\sum_{k=0}^{\infty}\lambda_k(1-\lambda_k)=+\infty$. Lin and Xu also comprehensively analyzed the convergence and convergence rate of the Picard iteration in \cite{lin2022convergence}.

Extensive research has been devoted to the approximation of fixed points of nonexpansive mappings using the KM iteration \eqref{TradKM}, e.g., \cite{goebel1984uniform,kanzow2017generalized,liang2016convergence,opial1967weak,yao2008weak}.
However, the KM iteration lacks inherent acceleration mechanisms, which limits its practical efficiency. To address these limitations, Bot and Nguyen introduced in \cite{bot2023fast} the Fast KM algorithm, which is formulated as:
\begin{align}
\label{FastKM}\bmx^{k+1}=&\left(1-\frac{s\alpha}{2(k+\alpha)}\right)\bmx^k+\frac{(1-s)k}{k+\alpha}(\bmx^k-\bmx^{k-1})\notag\\
&+\frac{s\alpha}{2(k+\alpha)}T\bmx^k+\frac{sk}{k+\alpha}(T\bmx^k-T\bmx^{k-1}),\ \ k\in\bbN_+,
\end{align}
where $T$ is a $\theta$-averaged nonexpansive operator, $\alpha>2$, $\bmx^0,\bmx^1\in\mH$, and $0<s\leq\frac{1}{\theta}$.  This method incorporates Nesterov-type momentum into the KM iteration and demonstrates the weak convergence of the algorithm.

In parallel to the effort to enhance the KM iteration, there has also been a resurgence of interest in an alternative classical scheme known as the Halpern iteration \cite{halpern1967fixed}. Unlike KM-type updates, the Halpern method incorporates a fixed anchor point to guide the sequence. The basic form is given by
\begin{equation}\label{tradHalpern}
\bmx^{k}=\lambda_k\bmx^0+(1-\lambda_k) T\bmx^{k-1},\ \ k\in\bbN_+,
\end{equation}
where $\lambda_k=\frac{1}{k+1}$. This classical iteration has recently attracted renewed attention \cite{he2024convergence,lieder2021convergence,park2022exact,qi2021convergence,yoon2021accelerated,zhang2025hppp}. In particular, He et al. \cite{he2024convergence} proposed an adaptive way to choose the anchoring parameter $\lambda_k$, defined by $\lambda_k=\frac{1}{\varphi_k+1}$, where $\varphi_k=\frac{2\langle \bmx^{k-1}-T\bmx^{k-1},\bmx^{0}-\bmx^{k-1}\rangle}{\|\bmx^{k-1}-T\bmx^{k-1}\|^2}+1$. The Halpern iteration requires a fixed anchor point throughout the iteration process, and the choice of this anchor point can significantly affect the convergence behavior. Besides, its acceleration is not significant for some relatively complex real-world applications, such as image denoising and low-rank matrix completion problems.

In this work, we propose a theoretically guaranteed fixed-point type algorithm, Two-step KM Algorithm (TKMA). Specifically, we assume that $T$ is an $\alpha$-averaged nonexpansive operator with $\alpha\in(0,1)$. Inspired by the geometric structures of the averaged nonexpansiveness and the KM acceleration, as well as the information from the composite operator $T^2$, we develop an adaptive momentum scheme for the KM iteration. To further exploit higher-order information, the proposed TKMA is designed as a convex combination of the KM iteration with adaptive momentum and the Picard iteration of the composite operator $T^2$. The resulting iterative scheme is given by:
\begin{equation}\label{TKMA}
\bmx^{k+1}=(1-t)T^2\bmx^{k}+tT^{\theta_k}\bmx^{k},\ \ k\in\bbN,
\end{equation}
where $t\in(0,1)$, $\theta_k:=\frac{-\langle T\bmx^k-\bmx^k,T\bmx^k-T^2\bmx^k\rangle}{\|T\bmx^k-\bmx^k\|^2}$ and
$$
T^{\theta_k}:=T+\theta_k(T-\mI).
$$
We establish the weak convergence of the TKMA and its $o(1/\sqrt{k})$ convergence rate in terms of the distance between consecutive iterates under certain conditions.
The main contributions of this paper can be summarized as follows:
\begin{itemize}
\item[$(i)$] We propose a Two-step Krasnosel'ski\u{\i}-Mann Algorithm (TKMA). In addition, we theoretically prove the weak convergence of TKMA and demonstrate that when $\alpha\in(0,\frac{1}{2}]$ and the adaptive momentum parameter $\theta_k$ in the KM iteration satisfies $\sum_{k=0}^{\infty} |\theta_{k+1}-\theta_k|<+\infty$, the TKMA achieves an $o(1/\sqrt{k})$ convergence rate in terms of the distance between consecutive iterates.
\item[$(ii)$] We rigorously prove that the KM momentum accelerates convergence for averaged nonexpansive operators with appropriate momentum parameters. Building on the geometric properties of such operators and the acceleration of KM momentum, along with information from the operator $T^2$, we propose an adaptive momentum scheme for the KM iteration.
\item[$(iii)$] We applied the proposed TKMA to image denoising and low-rank matrix completion problems, demonstrating its high efficiency.
\end{itemize}

The rest of this paper presents the following contents. In section II, we present fixed-point algorithms for solving two classes of two-term optimization problems. In section III, we propose a new adaptive selection scheme for the momentum parameters in KM iteration, and then develop the TKMA. We analyze in section IV the convergence and convergence rate of the TKMA. Section V presents the numerical results for comparison of the proposed TKMA with the FPPA, the PGA, the Fast KM algorithm and the Halpern algorithm in image denoising and low-rank matrix completion problems. Section VI offers a conclusion.

\section{Fixed-point formulations for optimization}\label{FFO}
In this section, we introduce fixed-point algorithms designed to solve two classes of two-term optimization problems. We denote by $\Gamma_{0}(\mH)$ the class of all proper lower semicontinuous convex functions from $\mH$ to $[-\infty,+\infty]$. Firstly, we consider a class of two-term optimization problems of the form
\begin{equation}\label{twominfunc}
\argmin_{\bmx\in \mH} f(\bmx)+g(\bmx),
\end{equation}
where $f\in\Gamma_{0}(\mH)$ is differentiable and $g\in\Gamma_{0}(\mH)$ may not be differentiable. This type of optimization problems are raised from machine learning (e.g. $\ell_1$-SVM, LASSO regression)\cite{li2019two}, compressed sensing \cite{figueiredo2008gradient} and image processing \cite{beck2009fast,figueiredo2003algorithm}. Throughout this paper, we assume that the objective function $f+g$ has at least one minimizer without further mentioning. Solutions of problem \eqref{twominfunc} may be reformulated as fixed points of certain operators, depending on the smoothness of the objective function $f+g$. To this end, we first recall the notions of the proximity operator and subdifferential of a convex function. For $\psi\in\Gamma_{0}(\mH)$, the proximity operator of $\psi$ at $\bmx\in\mH$ is defined by
\[
\prox_{\psi}(\bmx):=\argmin_{\bmu\in\mH}\left\{\frac{1}{2}\Vert \bmu-\bmx\Vert^2+\psi(\bmu)\right\},
\]
and the subdifferential of $\psi$ at $\bmx\in\mH$ is defined by
$$
\partial\psi(\bmx):=\{\bmy\in\mH: \psi(\bmz)\geq\psi(\bmx)+\langle \bmy, \bmz-\bmx\rangle \  \text{for all} \ \bmz\in\mH\}.
$$
When the proximity operator of function $g$ has a closed form or can be easily computed, by using Fermat's rule (Theorem 16.3 of \cite{bauschke2017convex}) and a relation between the subdifferential and the proximity operator (Proposition 2.6 of \cite{micchelli2011proximity}), a minimizer of \eqref{twominfunc} is identified as a fixed point of operator
\begin{equation*}
T_1:=\prox_{\beta g}\circ(\mI-\beta\nabla f), \ \text{where} \ \beta>0.
\end{equation*}
In contrast, for some optimization problems of the form  \eqref{twominfunc}, the proximity operator of $g$ does not have a closed form. A notable example is the total variation (TV) regularized image denoising problem:
\begin{equation}\label{TVdenoise}
\argmin_{\bmu\in\bbRd} \frac{1}{2}\|\bmu-\bmx\|^2+\phi(\bmB\bmu),
\end{equation}
where $\bmu\in\bbRd$ represents the denoised image, $\bmx\in\bbRd$ denotes the observed noisy image, $\phi\circ \bmB$ is the TV regularization term, $\phi$ is  defined as $\phi(\cdot):=\mu\|\cdot\|_1$, $\mu>0$, and the matrix $\bmB\in\bbR^{2d\times d}$ is the first-order difference matrix, whose specific definition can be found in Section \ref{ImageDenoising}. Specifically, the proximity operator of function $\phi$ in problem \eqref{TVdenoise} has a closed form, while $\phi\circ \bmB$ does not. To derive the fixed-point algorithm for solving this problem, we need to use Fermat's rule and Theorem 6.39 of \cite{beck2017first}, which can be rewritten as the following lemma.
\begin{lemma}\label{secproxthm}
Let $\psi\in\Gamma_{0}(\mH)$ and $\bmx,\bmy\in\mH$. Then
\begin{equation*}
\bmy\in\partial \psi(\bmx) \ \text{if and only if} \ \bmx=\prox_\psi(\bmx+\bmy)
\end{equation*}
\end{lemma}

We next summarize the derivation of the Fixed-Point Proximity Algorithm (FPPA) proposed in \cite{micchelli2011proximity} to solve problem \eqref{TVdenoise}. By Fermat's rule, $\bmu$ is a solution of model \eqref{TVdenoise} if and only if
\begin{equation*}\label{partial}
\bm0\in(\bmu-\bmx)+\bmB^{\top}\partial\phi(\bmB\bmu),
\end{equation*}
Then there exists
\begin{equation}\label{parphiBu}
\lambda \bmy\in\partial\phi(\bmB\bmu),\ \lambda>0
\end{equation}
such that
\begin{equation*}\label{partial1}
\bm0=(\bmu-\bmx)+\lambda \bmB^{\top}\bmy,
\end{equation*}
that is,
\begin{equation}\label{partial2}
\bmu=\bmx-\lambda \bmB^{\top}\bmy.
\end{equation}
In addition, \eqref{parphiBu} is equivalent to
\begin{equation}\label{partial3}
\bmy\in\partial\left(\frac{1}{\lambda}\phi\right)(\bmB\bmu).
\end{equation}
Using Lemma \ref{secproxthm} for \eqref{partial3}, we have
\begin{equation*}\label{parprox}
\bmB\bmu=\prox_{\frac{1}{\lambda}\phi}(\bmB\bmu+\bmy),
\end{equation*}
which is equivalent to
\begin{equation}\label{Iprox}
\bmy=\left(\mI-\prox_{\frac{1}{\lambda}\phi}\right)(\bmB\bmu+\bmy).
\end{equation}
Substituting \eqref{partial2} into \eqref{Iprox} gives
\begin{equation*}\label{FPPAfordeno}
\bmy=\left(\mI-\prox_{\frac{1}{\lambda}\phi}\right)\left(\bmB\bmx+(I-\lambda \bmB\bmB^{\top})\bmy\right).
\end{equation*}
Let operators $\mA:\bbR^{2d}\to\bbR^{2d}$ and $T_2:\bbR^{2d}\to\bbR^{2d}$ be defined by
$$
\mA(\bmy):=\bmB\bmx+(I-\lambda \bmB\bmB^{\top})\bmy,\ \ \bmy\in\bbR^{2d}
$$
and
\begin{equation}\label{operatorTfordenoising}
T_2:=\left(\mI-\prox_{\frac{1}{\lambda}\phi}\right)\circ\mA,
\end{equation}
respectively. Then the solution of model \eqref{TVdenoise} can be represented by
\begin{equation*}
\bmu=\bmx-\lambda \bmB^{\top}\bmy^*,
\end{equation*}
where $\bmy^*$ is a fixed point of operator $T_2$.

\section{Two-step Krasnosel'ski\u{\i}-Mann Algorithm}\label{sec3}
In this section, we propose a Two-step Krasnosel'ski\u{\i}-Mann Algorithm (TKMA) inspired by the geometric properties of the averaged nonexpansive operator and the KM acceleration technique.
To this end, we give the definitions of the nonexpansiveness and averaged-nonexpansiveness.
We say that $T$ is nonexpansive if $\|T\bmx-T\bmy\|\leq\|\bmx-\bmy\|$ holds for all $\bmx,\bmy\in\mH$, and $T$ is $\alpha$-averaged nonexpansive if there exists a nonexpansive operator $\mN:\mH\rightarrow\mH$ such that $T=(1-\alpha)\mI+\alpha\mN$, where $\alpha\in(0,1)$.

The KM iteration \eqref{TradKM} can be viewed as a generalization of fixed-point iteration with momentum acceleration. This momentum scheme enhances the approximation of the solution by adding the current fixed-point update to the difference between the current and previous updates. Specifically, for $\theta\in\bbR$ and operator $T:\mH\to\mH$, we define
\begin{equation}\label{def:Ttheta}
T^{\theta}:=T+\theta(T-\mI).
\end{equation}
The KM iteration \eqref{TradKM} can then be rewritten as
\begin{equation}\label{eq:KMiter}
\bmx^{k+1}=T^{\theta_k}\bmx^k,\ \ k\in\bbN,
\end{equation}
where $\{\theta_k\}\subset\bbR$ are the momentum step-sizes. The selection of these step-sizes is crucial not only for efficiency but also for ensuring convergence. Suppose that operator $T$ is $\alpha$-averaged nonexpansive, i.e., $T=(1-\alpha)\mI+\alpha\mN$, where $\mN$ is a nonexpansive operator. Then the KM iteration \eqref{eq:KMiter} of $T$ can be rewritten as
\begin{equation}\label{KMaveriter}
\bmx^{k+1}=[1-\alpha(1+\theta_k)]\bmx^k+\alpha(1+\theta_k)\mN \bmx^k,\ \ k\in\bbN.
\end{equation}
In order to explore the convergence conditions of the iteration \eqref{KMaveriter}, we introduce the well-known KM theorem. We denote by $\Fix(T)$ the set of all fixed points of operator $T$.
\begin{lemma}[Krasnosel'ski\u{\i}-Mann theorem \cite{groetsch1972note}]\label{lem:KMthm}
Let $\mN:\mH\to\mH$ be a nonexpansive operator such that $\Fix(\mN)\neq\varnothing$. For $\{\alpha_k\}\subset[0,1]$ and $\bmx^0\in\mH$, define
$$
\bmx^{k+1}:=(1-\alpha_k)\bmx^k+\alpha_k\mN \bmx^k,\ \ k\in\bbN.
$$
If $\sum_{k=0}^{\infty}\alpha_k(1-\alpha_k)=+\infty$, then the sequence $\{\bmx^k\}$ converges weakly to a point of $\Fix(\mN)$.
\end{lemma}

According to Lemma \ref{lem:KMthm}, the KM iteration \eqref{KMaveriter} weakly converges if
\begin{equation*}
\{\theta_k\}\in\left[-1,\frac{1-\alpha}{\alpha}\right]\ \text{and} \ \sum_{k=0}^\infty(1+\theta_k)[1-\alpha(1+\theta_k)]=+\infty
\end{equation*}
hold. In fact, we can demonstrate that for appropriate momentum step-sizes, the KM iteration achieves accelerated convergence compared to the original fixed-point iteration of $T$. Here, we explore an adaptive selection scheme for the KM momentum parameters, inspired by the geometric structure of averaged nonexpansive operators and the KM acceleration. Suppose that $T$ is a $\frac{1}{4}$-averaged nonexpansive operator and let $\hx\in\Fix(T)$. For a given vector $\bmx$, define $\bmx_{1/2}:=\frac{1}{2}\hx+\frac{1}{2}\bmx$ and $\bmx_{1/4}:=\frac{1}{4}\hx+\frac{3}{4}\bmx$. The range of $T\bmx$, i.e., the region where the output vector can fall after $T$ acts on $\bmx$, is then the ball centered at $\bmx_{1/4}$ with radius $\left\|\bmx-\bmx_{1/4}\right\|$. As illustrated in Figure \ref{fig:KMacceler} and by \eqref{def:Ttheta}, the KM momentum accelerates the fixed-point iteration of a $\frac{1}{4}$-averaged nonexpansive operator. This figure shows that when the vector $T^{\theta}\bmx - \bmx$ is orthogonal to $T^{\theta}\bmx - \hx$, the point $T^{\theta}\bmx$ minimizes the distance to $\hx$ on the line through $\bmx$ and $T\bmx$. In this case, by the Pythagorean theorem, we get that the optimal momentum step-size $\theta^*$ is given by
\begin{equation}\label{thetastar}
\theta^*:=\frac{-\langle T\bmx-\bmx,T\bmx-\hat{\bmx}\rangle}{\|T\bmx-\bmx\|^2}.
\end{equation}
\begin{figure}[H]
\centering
\includegraphics[width=0.4\textwidth]{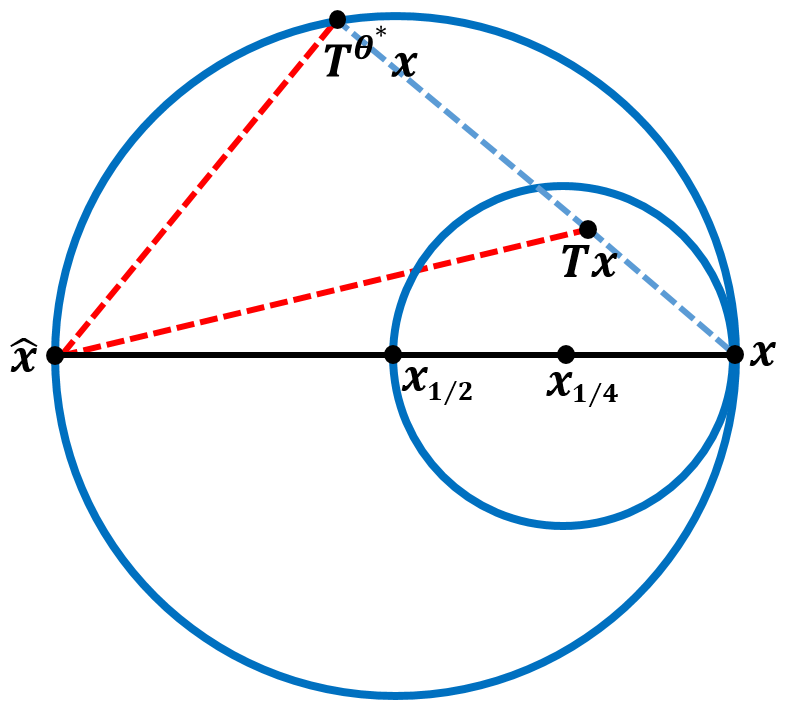}
\caption{KM momentum acceleration for the fixed-point iteration of $\frac{1}{4}$-averaged nonexpansive operators.}
\label{fig:KMacceler}
\end{figure}

In the remainder of this paper, we assume that operator $T$ has at least one fixed point. We shall show that with the step size in \eqref{thetastar}, the KM iteration converges faster to the fixed point than the standard fixed-point iteration. To this end, we first recall the Proposition 4.35 of \cite{bauschke2017convex} as the following Lemma \ref{lem:equivAN}, which provide an equivalent condition of averaged nonexpansiveness.

\begin{lemma}\label{lem:equivAN}
Let $\alpha\in(0,1)$. Then operator $T:\mH\to\mH$ is $\alpha$-averaged nonexpansive if and only if for all $\bmx,\bmy\in\mH$,
\begin{equation*}\label{neq:equivAN}
\|T\bmx-T\bmy\|^2+\frac{1-\alpha}{\alpha}\|(\mI-T)\bmx-(\mI-T)\bmy\|^2\leq\|\bmx-\bmy\|^2.
\end{equation*}
\end{lemma}

\begin{theorem}\label{thm_KMaccel}
Let $\alpha\in(0,\frac{1}{2})$, $T$ be an $\alpha$-averaged nonexpansive and $\hx\in\Fix(T)$. For any given $\bmx\in\mH\backslash\Fix(T)$, define $\theta^*:=\frac{-\langle T\bmx-\bmx,T\bmx-\hat{\bmx}\rangle}{\|T\bmx-\bmx\|^2}$, then $\theta^*>0$ and $\|T^{\theta}\bmx-\hat{\bmx}\|<\|T\bmx-\hat{\bmx}\|$ for $\theta\in(0,2\theta^*)$. Moreover, function $\psi(\theta):=\|T^{\theta}\bmx-\hat{\bmx}\|$ is strictly decreasing on $[0,\theta^*]$.
\end{theorem}
\begin{proof}
Note that
\begin{equation}\label{eq:Txxhatxnorm}
\|T\bmx-\bmx\|^2+\|T\bmx-\hx\|^2-\|\bmx-\hx\|^2=2\langle T\bmx-\bmx,T\bmx-\hx\rangle.
\end{equation}
The $\alpha$-averaged nonexpansiveness together with Lemma \ref{lem:equivAN} gives
\begin{equation}\label{neq:TxxhatxAN}
\|T\bmx-\hx\|^2-\|\bmx-\hx\|^2\leq-\frac{1-\alpha}{\alpha}\|T\bmx-\bmx\|^2.
\end{equation}
Combining \eqref{eq:Txxhatxnorm} and \eqref{neq:TxxhatxAN} yields
\begin{align*}
-\langle T\bmx-\bmx,T\bmx-\hx\rangle&\geq-\frac{1}{2}\left(\|T\bmx-\bmx\|^2-\frac{1-\alpha}{\alpha}\|T\bmx-\bmx\|^2\right)\\
&=\left(\frac{1}{2\alpha}-1\right)\|T\bmx-\bmx\|^2,
\end{align*}
which implies that $\theta^*\geq\frac{1}{2\alpha}-1>0$ since $\alpha\in(0,\frac{1}{2})$ and $\|T\bmx-\bmx\|>0$. Then for $\theta\in(0,2\theta^*)$, we have $|\theta^*-\theta|<|\theta^*|$. Hence
\begin{align*}
\|T^{\theta^*}\bmx-T^{\theta}\bmx\|&=|\theta^*-\theta|\cdot\|T\bmx-\bmx\|\\
&<|\theta^*|\cdot\|T\bmx-\bm\bmx\|=\|T^{\theta^*}\bmx-T\bmx\|
\end{align*}
It is easy to verify that
$$
\langle T^{\theta^*}\bmx-\hat{\bmx},T^{\theta^*}\bmx-T^{\theta}\bmx\rangle=0,\ \ \mbox{for all}\ \theta\in\bbR.
$$
Now by the Pythagorean theorem, for $\theta\in(0,2\theta^*)$,
\begin{align*}
\|T^{\theta}\bmx-\hat{\bmx}\|&=\sqrt{\|T^{\theta^*}\bmx-\hat{\bmx}\|^2+\|T^{\theta^*}\bmx-T^{\theta}\bmx\|^2}\\
&<\sqrt{\|T^{\theta^*}\bmx-\hat{\bmx}\|^2+\|T^{\theta^*}\bmx-T\bmx\|^2}\\
&=\|T\bmx-\hat{\bmx}\|.
\end{align*}
Note that
\begin{align*}
\psi(\theta)&=\sqrt{\|T^{\theta^*}\bmx-\hat{\bmx}\|^2+\|T^{\theta^*}\bmx-T^{\theta}\bmx\|^2}\\
&=\sqrt{\|T^{\theta^*}\bmx-\hat{\bmx}\|^2+(\theta^*-\theta)^2\|T\bmx-\bmx\|^2},
\end{align*}
which together with the fact $\|T\bmx-\bmx\|>0$ implies that $\psi$ is strictly decreasing on $[0,\theta^*]$. This completes the proof.
\end{proof}

From the above theorem, we know that the convergence speed increases as the momentum parameter $\theta$ approaches $\theta^*$ in each iteration. We denote by $O(\bmx,r)$ the ball with center $\bmx$ and radius $r$, where $\bmx\in\mH$, $r\in\bbR_+$. For the case $\alpha\in\big[\frac{1}{2},1\big)$, analogous to the proof of Theorem \ref{thm_KMaccel}, we can establish the following result.
\begin{theorem}\label{thm_KMaccel1}
Let $\alpha\in\big[\frac{1}{2},1\big)$, $T$ be an $\alpha$-averaged nonexpansive and $\hx\in\Fix(T)$. For given $\bmx\in\mH\backslash\Fix(T)$, define $\theta^*:=\frac{-\langle T\bmx-\bmx,T\bmx-\hx\rangle}{\|T\bmx-\bmx\|^2}$. If $T\bmx\in O(\hx_{1/2},\|\bmx-\hx_{1/2}\|)$, then $\theta^*\geq0$ and $\|T^{\theta}\bmx-\hx\|\leq\|T\bmx-\hx\|$ for $\theta\in[0,2\theta^*]$. Otherwise, $\theta^*<0$ and $\|T^{\theta}\bmx-\hx\|\leq\|T\bmx-\hx\|$ for $\theta\in[2\theta^*,0]$.
\end{theorem}

According to Theorem \ref{thm_KMaccel} and \ref{thm_KMaccel1}, we know that setting the KM parameter as $\theta^*$ yields best approximation to the fixed point $\hx$ along the momentum direction. Of course, the expression for $\theta^*$ involves $\hx$, the desired solution, which is naturally unknown. To derive a practical KM algorithm, we replace $\hx$ with the approximation $T^2\bmx:=T(T(\bmx))$. This term leverages information from a two-step fixed-point iteration and is closer to $\hx$ than the point $T\bmx$. Then we propose the following generalized iterative scheme of the Two-step KM Algorithm (TKMA):

\begin{equation}\label{eq:IKMAiter}
\bmx^{k+1}=(1-t)T^2\bmx^k+t T^{\theta_k}\bmx^k,\ \ k\in\bbN,
\end{equation}
where $t\in(0,1)$,
\begin{equation}\label{def:Tthetak}
T^{\theta_k}:=T+\theta_k(T-\mI),
\end{equation}
and
\begin{equation}\label{IKMAthetak}
\theta_k:=\frac{-\langle T\bmx^k-\bmx^k,T\bmx^k-T^2\bmx^k\rangle}{\|T\bmx^k-\bmx^k\|^2}.
\end{equation}
To make the theoretical analysis more concise, we always assume that $T\bmx^k\neq\bmx^k$ to ensure the well-definedness of $\theta_k$. If
$T\bmx^k=\bmx^k$, it is clear that the algorithm has already found a fixed point, which is the optimal solution to the original problem, and thus no further iterations are needed.

\section{Convergence and convergence rate analysis}\label{sec4}
In this section, we always let $T$ be an operator from $\mH$ to $\mH$, $\hat{\bmx}$ be a fixed point of $T$, and let $\{\bmx^k\}\subset\mH$ be a sequence generated by algorithm \eqref{eq:IKMAiter}. We then analyze the weak convergence and convergence rate of the TKMA. To this end, we first derive the range of the parameter $\theta$ and investigate the nonexpansiveness of operator $T^{\theta}$ in the following proposition.

\begin{proposition}\label{prop:thetarange}
Suppose that $\alpha\in(0,1)$, $T$ is an $\alpha$-averaged nonexpansive operator and $T^{\theta}$ is defined by \eqref{def:Ttheta}. For any $\bmx\in \mH\backslash Fix(T)$, let $\theta:=\frac{-\langle T\bmx-\bmx,T\bmx-T^2\bmx\rangle}{\|T\bmx-\bmx\|^2}$. Then the following facts hold:
\begin{itemize}
\item[$(i)$] $\theta\in[1-2\alpha,1]$;
\item[$(ii)$] If $\alpha\in(0,1/2]$, then $T^{\theta}$ is nonexpansive.
\end{itemize}
\end{proposition}
\begin{proof}
It follows from the definition of $\theta$ that
$$
\theta\|\bmx-T\bmx\|^2=\langle \bmx-T\bmx,T\bmx-T^2\bmx\rangle,
$$
and hence
\begin{align}
\notag&\|T\bmx-T^2\bmx\|^2+(1-\alpha)\|\bmx-T\bmx\|^2-2(1-\alpha)\theta\|\bmx-T\bmx\|^2\\
\label{eq:xTxTxT2xsquare}=&\alpha\|T\bmx-T^2\bmx\|^2+(1-\alpha)\|(\bmx-T\bmx)-(T\bmx-T^2\bmx)\|^2.
\end{align}
Since $T$ is $\alpha$-averaged nonexpansive, by Lemma \ref{lem:equivAN}, we get that
$$
\|T\bmx-T^2\bmx\|^2+\frac{1-\alpha}{\alpha}\|(\bmx-T\bmx)-(T\bmx-T^2\bmx)\|^2\leq\|\bmx-T\bmx\|^2,
$$
which together with \eqref{eq:xTxTxT2xsquare} yields
\begin{equation}\label{thetaineq}
\|T\bmx-T^2\bmx\|^2-2(1-\alpha)\theta\|\bmx-T\bmx\|^2\leq(2\alpha-1)\|\bmx-T\bmx\|^2.
\end{equation}
Using the Cauchy-Schwarz inequality, we have
$$
\langle \bmx-T\bmx,T\bmx-T^2\bmx\rangle\leq\|T\bmx-T^2\bmx\|\|\bmx-T\bmx\|.
$$
Thus
\begin{equation}\label{neq:TxT2xns}
\theta^2\|\bmx-T\bmx\|^2=\frac{\langle \bmx-T\bmx,T\bmx-T^2\bmx\rangle^2}{\|T\bmx-\bmx\|^2}\leq\|T\bmx-T^2\bmx\|^2.
\end{equation}
Combining \eqref{thetaineq} and \eqref{neq:TxT2xns}, we see that
\begin{align*}
\theta^2\|\bmx-T\bmx\|^2-2(1-\alpha)\theta\|\bmx-T\bmx\|^2\leq(2\alpha-1)\|\bmx-T\bmx\|^2.
\end{align*}
Then
\begin{align*}
\theta^2-2(1-\alpha)\theta\leq(2\alpha-1),
\end{align*}
that is,
$$
\left(\theta-(1-\alpha)\right)^2\leq\alpha^2,
$$
which implies that item $(i)$ holds.

We then consider item $(ii)$. Since $T$ is $\alpha$-averaged nonexpansive, $\alpha\in(0,1)$, there exists a nonexpansive operator $\mN$ such that $T=(1-\alpha)\mI+\alpha\mN$, and hence $T^{\theta}=(1+\theta)T-\theta\mI=\big[1-\alpha(1+\theta)\big]\mI+\alpha(1+\theta)\mN$. It has been shown in item $(i)$ that $\theta\in[1-2\alpha,1]$. If $\alpha\in(0,1/2]$, we have $\alpha(1+\theta)\in(0,1]$, and hence $T^{\theta}$ is  nonexpansive. This completes the proof.
\end{proof}

In the following lemma, we prove the monotonicity of sequence $\{\|\bmx^{k}-\hat{\bmx}\|\}$.

\begin{lemma}\label{lem:monotonic}
Let $T$ be an $\alpha$-averaged nonexpansive operator, $\alpha\in(0,1)$, and let $\{\bmx^k\}$ be a sequence generated by \eqref{eq:IKMAiter}. If either $\alpha\in(0,1/2]$, $t\in(0,1)$ or $\alpha\in(1/2,1)$, $t\in\left(0,\frac{(1-\alpha)^2+(1-\alpha)\sqrt{(1-\alpha)^2+8\alpha^2}}{4\alpha^2}\right)$ holds, then there exists $\eta>0$ such that the following fact holds:
\begin{equation}\label{xkbounded}
\|\bmx^{k+1}-\hat{\bmx}\|^2\leq\|\bmx^k-\hat{\bmx}\|^2-\eta\|T\bmx^k-\bmx^k\|^2,\ \ \forall\ \hat{\bmx}\in\Fix(T), \ k\in\bbN.
\end{equation}
\end{lemma}
\begin{proof}
For any $k\in\bbN$, we let $\bmu^k:=T\bmx^k-\bmx^k$, $\bmv^k:=T^2\bmx^k-T\bmx^k$, $\bmw^k:=\bmx^k-\hat{\bmx}$. Then $T\bmx^k-\hat{\bmx}=\bmu^k+\bmw^k$, $T^2\bmx^k-\hat{\bmx}=\bmu^k+\bmv^k+\bmw^k$, and $T^{\theta_k}\bmx^k-\hat{\bmx}=(1+\theta_k)\bmu^k+\bmw^k$.

We first consider the case $\alpha\in(0,1/2]$ and $t\in(0,1)$. For any $\hat{\bmx}\in\Fix(T)$ and $k\in\bbN$, it follows from the $\alpha$-averaged nonexpansiveness of $T$ and Lemma \ref{lem:equivAN} that
\begin{equation}\label{neq:T2xkxhatnorm2}
\|\bmu^k+\bmv^k+\bmw^k\|^2\leq\|\bmu^k+\bmw^k\|^2\leq\|\bmw^k\|^2-\frac{1-\alpha}{\alpha}\|\bmu^k\|^2.
\end{equation}
Of course, $\|\bmu^k+\bmv^k+\bmw^k\|\leq\|\bmw^k\|$. Note that $\Fix\left(T^{\theta}\right)=\Fix\left(T\right)$. By item $(ii)$ of Proposition \ref{prop:thetarange}, we have
\begin{equation}\label{neq:nonexpTthetak}
\|(1+\theta_k)\bmu^k+\bmw^k\|\leq\|\bmw^k\|.
\end{equation}
Then
\begin{equation}\label{neq:T2xkTthetakxk}
\langle \bmu^k+\bmv^k+\bmw^k,(1+\theta_k)\bmu^k+\bmw^k\rangle\leq\|\bmu^k+\bmv^k+\bmw^k\|\cdot\|(1+\theta_k)\bmu^k+\bmw^k\|\leq\|\bmw^k\|^2.
\end{equation}
According to the definition of $\bmx^{k+1}$ in \eqref{eq:IKMAiter}, and then using \eqref{neq:nonexpTthetak} and \eqref{neq:T2xkTthetakxk}, we have
\begin{align}\label{neq:xkmonoto}
\|\bmx^{k+1}-\hat{\bmx}\|^2&=\left\|(1-t)(\bmu^k+\bmv^k+\bmw^k)+t((1+\theta_k)\bmu^k+\bmw^k)\right\|^2\notag\\
&\leq(1-t)^2\|\bmu^k+\bmv^k+\bmw^k\|^2+(2t-t^2)\|\bmw^k\|^2.
\end{align}
Let $\eta:=\frac{(1-t)^2(1-\alpha)}{\alpha}>0$. Combining \eqref{neq:xkmonoto} with \eqref{neq:T2xkxhatnorm2} readily yields the desired result.

We next deal with the case $\alpha\in(1/2,1)$ and $t\in\left(0,\bar{\alpha}\right)$, where
\begin{equation}\label{defalphabar}
\bar{\alpha}:=\frac{(1-\alpha)^2+(1-\alpha)\sqrt{(1-\alpha)^2+8\alpha^2}}{4\alpha^2}.
\end{equation}
Note that
\begin{align}\label{xkmonoto}
\|\bmx^{k+1}-\hat{\bmx}\|^2=&\|\bmw^k+(1+t\theta_k)\bmu^k+(1-t)\bmv^k\|^2\notag\\
=&\|\bmw^k\|^2+2(1+t\theta_k)\langle\bmw^k,\bmu^k\rangle+2(1-t)\langle\bmw^k,\bmv^k\rangle+(1+t\theta_k)^2\|\bmu^k\|^2\notag\\
&+2(1+t\theta_k)(1-t)\langle\bmu^k,\bmv^k\rangle+(1-t)^2\|\bmv^k\|^2.
\end{align}
Since $T$ is $\alpha$-averaged nonexpansive, by Lemma \ref{lem:equivAN}, we obtain that
\begin{equation}\label{averaged2}
\|\bmu^k+\bmv^k+\bmw^k\|^2\leq\|\bmu^k+\bmw^k\|^2-\frac{1-\alpha}{\alpha}\|\bmv^k\|^2.
\end{equation}
It follows from \eqref{neq:T2xkxhatnorm2} and \eqref{averaged2} that
\begin{equation}\label{neq:wuandwv}
\langle \bmw^k,\bmu^k\rangle\leq-\frac{1}{2\alpha}\|\bmu^k\|^2  \ , \ \ \ \langle \bmv^k,\bmw^k\rangle\leq-\frac{1}{2\alpha}\|\bmv^k\|^2-\langle \bmv^k,\bmu^k\rangle.
\end{equation}
Let $A_k:=(1+t\theta_k)\left(1+t\theta_k-\frac{1}{\alpha}\right)$, $B:=(1-t)\left(1-t-\frac{1}{\alpha}\right)$, $C_k:=2t\theta_k(1-t)$. It is easy to see that $\bar{\alpha}$ defined by \eqref{defalphabar} is in $(0,1)$. Since $\alpha\in(1/2,1)$ and $t\in\left(0,\bar{\alpha}\right)$, we know that $B<0$. Substituting \eqref{neq:wuandwv} into \eqref{xkmonoto} yields
\begin{align*}\label{xkmonoto1}
\|\bmx^{k+1}-\hat{\bmx}\|^2\leq&\|\bmw^k\|^2-\frac{1+t\theta_k}{\alpha}\|\bmu^k\|^2-\frac{1-t}{\alpha}\|\bmv^k\|^2-2(1-t)\langle\bmu^k,\bmv^k\rangle\notag\\
&+(1+t\theta_k)^2\|\bmu^k\|^2+2(1+t\theta_k)(1-t)\langle\bmu^k,\bmv^k\rangle+(1-t)^2\|\bmv^k\|^2\notag\\
=&\|\bmw^k\|^2+A_k\|\bmu^k\|^2+B\|\bmv^k\|^2+C_k\langle\bmu^k,\bmv^k\rangle.
\end{align*}
To prove \eqref{xkbounded}, it suffices to show that there exist $\eta>0$ such that
\begin{equation}\label{neq:ABC}
A_k\|\bmu^k\|^2+B\|\bmv^k\|^2+C_k\langle\bmu^k,\bmv^k\rangle\leq-\eta\|\bmu^k\|^2,\ \ \forall\ k\in\bbN.
\end{equation}
By the orthogonal decomposition theorem, there exists a unique $\lambda$ and $\bms$ satisfied $\langle\bmu^k,\bms\rangle=0$ such that $\bmv^k=\lambda\bmu^k+\bms$. Then we have
\begin{align*}\label{eq:ABC1}
A_k\|\bmu^k\|^2+B\|\bmv^k\|^2+C_k\langle\bmu^k,\bmv^k\rangle&=\left(A_k+B\lambda^2+C_k\lambda\right)\|\bmu^k\|^2+B\|\bms\|^2\notag\\
&\leq\left(A_k+B\lambda^2+C_k\lambda\right)\|\bmu^k\|^2.
\end{align*}
Let $h(\lambda):=A_k+B\lambda^2+C_k\lambda$. Since $B<0$, we know that $h(\lambda)$ is a concave function, and its maximum value is $h(-\frac{C_k}{2B})=A_k-\frac{C_k^2}{4B}$. To prove \eqref{neq:ABC}, it suffices to show that there exists $\eta>0$ such that $A_k-\frac{C_k^2}{4B}\leq-\eta$, $\forall k\in\bbN$.

Recall from item $(i)$ of Proposition \ref{prop:thetarange} that $\theta_k\in[1-2\alpha,1]$. The definitions of $A_k$ and $C_k$, together with the hypotheses $\alpha\in(1/2,1)$ and $0<t<\bar{\alpha}<1$, imply that both $A_k$ and $C_k^2$ attain their respective maximum at $\theta_k=1$. Moreover, since $B<0$ is independent of $\theta_k$, we can see that $A_k-\frac{C_k^2}{4B}$ is also maximized at $\theta_k=1$. In the case $\theta_k=1$,
$$
A_k-\frac{C_k^2}{4B}=(1+t)(t-a)+\frac{(1-t)t^2}{a+t},
$$
where $a:=\frac{1}{\alpha}-1>0$. Now we set $\eta:=-(1+t)(t-a)-\frac{(1-t)t^2}{a+t}$. Since $\alpha\in(1/2,1)$ and $t\in(0,\bar{\alpha})$, it is easy to verify that $\eta>0$, which implies that $A_k-\frac{C_k^2}{4B}\leq-\eta$, $\forall k\in\bbN$. This completes the proof.
\end{proof}

To prove the convergence of TKMA, we still need the following lemma.

\begin{lemma}\label{lem:converge}
Let $T$ be a nonexpansive operator and $\{\bmx^k\}$ be a sequence in $\mH$. If there exists an $\eta>0$ such that
\begin{equation}\label{xkdecreasing}
\|\bmx^{k+1}-\hat{\bmx}\|^2\leq\|\bmx^k-\hat{\bm\bmx}\|^2-\eta\|T\bmx^k-\bmx^k\|^2,\ \ \forall\ \hat{\bmx}\in\Fix(T),\ k\in\bbN,
\end{equation}
then $\{\bmx^k\}$ weakly converges to a fixed point of $T$.
\end{lemma}
\begin{proof}
It follows from \eqref{xkdecreasing} that $\{\|\bmx^k-\hat{\bmx}\|\}$ is monotonically decreasing. Of course, this sequence has a lower bound $0$. Thus, $\lim_{k\rightarrow\infty}\|\bmx^k-\hat{\bmx}\|$ exists for all $\hat{\bmx}\in\Fix(T)$, which implies that $\{\bmx^k\}$ is bounded. Hence there exists a subsequence $\{\bmx^{j_k}\}$ of $\{\bmx^k\}$ weakly converges to some $\tilde{\bmx}$. In addition, summing the inequality \eqref{xkdecreasing} for $k=0,1,\cdots,K$ and letting $K$ tend to $+\infty$, we obtain
\begin{align*}\label{Seriesineq}
\eta\sum_{k=0}^{\infty}\|T\bmx^{k}-\bmx^{k}\|^2\leq\|\bmx^0-\hx\|^2.
\end{align*}
Since $\eta>0$, the above inequality yields that $\lim_{k\rightarrow\infty}\|T\bmx^{k}-\bmx^{k}\|=0$,
which implies  $\lim_{k\rightarrow\infty}\|T\bmx^{j_k}-\bmx^{j_k}\|=0$. Then we know from Corollary 4.28 of \cite{bauschke2017convex} that $\tilde{\bmx}\in \Fix(T)$. Now according to Lemma 2.47 of \cite{bauschke2017convex}, we see that $\{\bmx^k\}$ weakly converges to a fixed-point of $T$, which completes the proof.
\end{proof}

\begin{theorem}\label{thm:conv}
Let $T$ be an $\alpha$-averaged nonexpansive operator, where $\alpha\in(0,1)$, and let $\{\bmx^k\}$ be a sequence generated by \eqref{eq:IKMAiter}. If either $\alpha\in(0,1/2]$, $t\in(0,1)$ or $\alpha\in(1/2,1)$, $t\in\left(0,\frac{(1-\alpha)^2+(1-\alpha)\sqrt{(1-\alpha)^2+8\alpha^2}}{4\alpha^2}\right)$ holds, then $\{\bmx^k\}$ weakly converges to a fixed point of $T$.
\end{theorem}
\begin{proof}
We shall employ Lemma \ref{lem:converge} to prove this theorem. Of course, the averaged nonexpansiveness of $T$ implies its nonexpansiveness. It has been shown in Lemma \ref{lem:monotonic} that \eqref{xkdecreasing} holds. Therefore, this theorem follows from Lemma \ref{lem:converge} immediately.
\end{proof}

Next, we analyze the convergence rate of the TKMA. To this end, we extend Lemma 3 of \cite{davis2017convergence} to a more general form, presented in the following Lemma \ref{generalconvergerate}. This generalized lemma is directly applicable to our analysis.

\begin{lemma}\label{generalconvergerate}
Let $\{a_k\}\subset\bbR$ and $\{b_k\}\subset\bbR$ be two nonnegative sequences such that $\sum_{k=0}^{\infty}a_k b_k<+\infty$ and $\sum_{k=0}^{\infty}a_k=+\infty$, and let $\{c_k\}\subset\bbR$ be a positive sequence such that $\lim_{k\rightarrow\infty} c_k=1$ and $\left|\sum_{k=0}^{\infty} \ln c_k\right|<+\infty$. If $b_{k+1}\leq c_kb_k$ for $k\in\bbN$, then $b_k=o\left(\frac{1}{A_k}\right)$, where $A_k:=\sum_{i=0}^{k}a_i$.
\end{lemma}
\begin{proof}
We first prove the boundedness of the sequence $\{A_kb_k\}$. Since $\left|\sum_{k=0}^{\infty} \ln c_k\right|<+\infty$, there exists $M_1>0$ such that $|\sum_{k=0}^{n} \ln c_k|\leq M_1$ for all $n\in\bbN$. Let $S_n:=\sum_{k=0}^{n} \ln c_k$ and $M_2:=e^{2M_1}$. We have $|S_n|\leq M_1$ and $M_2>1$. Hence, for any $m<n$,
\begin{equation*}
\prod_{k=m}^{n-1}c_k=\frac{\prod_{k=0}^{n-1}c_k}{\prod_{k=0}^{m-1}c_k}=e^{S_{n-1}-S_{m-1}}\leq e^{|S_{n-1}-S_{m-1}|}\leq M_2.
\end{equation*}
Since $b_{n+1}\leq c_nb_n$, by induction, we get that for any $m<n$,
\begin{equation*}\label{bnbound}
b_{n}\leq b_m\prod_{k=m}^{n-1}c_k\leq b_m M_2.
\end{equation*}
The non-negativity of $\{a_k\}$ and the fact $M_2>1$ yield that for any $k\leq n$,
\begin{equation}\label{akbkineq}
a_{k}b_n\leq M_2a_{k}b_{k}.
\end{equation}
Since $\sum_{k=0}^{\infty}a_k b_k<+\infty$, there exists $M_3>0$ such that $\sum_{k=0}^{n}a_k b_k\leq M_3$ for all $n\in\bbN$. Summing \eqref{akbkineq} for $k=0,\cdots,n$, we obtain that
\begin{equation*}
A_nb_n\leq M_2\sum_{k=0}^{n}a_{k}b_{k}\leq M_2M_3,
\end{equation*}
which implies that the nonnegative sequence $\{A_kb_k\}$ is bounded.

The fact $\sum_{k=0}^{\infty}a_k b_k<+\infty$ also implies that for any $\varepsilon>0$, there exists $N\in\bbN$ such that
\begin{equation}\label{akbkleq}
\sum_{k=N+1}^{n}a_kb_k<\varepsilon,
\end{equation}
for all $n>N$. The boundedness of $\{A_kb_k\}$ implies the existence of $\limsup\limits_{k\to\infty}A_kb_k<+\infty$, and we denote it by $\delta$. Of course, $\delta\geq 0$.

We now prove that $\delta=0$. Suppose, to get a contradiction, that $\delta>0$. Then there exists a subsequence $\{n_j\}$ of $\{n\}$ such that
\begin{equation}\label{geq:Anjbnj}
A_{n_j}b_{n_j}\geq\frac{\delta}{2}
\end{equation}
Since $\lim_{n\to\infty}A_n=+\infty$, we have $\lim_{j\to\infty}A_{n_j}=+\infty$, which means that for any $N'\in\bbN$, there exists sufficiently large $J$ such that $A_{n_j}>2A_{N'}$, that is,
\begin{equation}\label{AnAN}
A_{n_j}-A_{N'}\geq\frac{1}{2}A_{n_j}
\end{equation}
for all $j\geq J$ and $n_j>N'$. It follows from \eqref{akbkineq} that for any $k\leq n_j$,
\begin{equation}\label{akbkineq1}
a_{k}b_{n_j}\leq M_2a_{k}b_{k}.
\end{equation}
Summing \eqref{akbkineq1} for $k=N'+1,\cdots,n_j$, we obtain that
\begin{equation*}
\sum_{k=N'+1}^{n_j}a_kb_{n_j}\leq M_2\sum_{k=N'+1}^{n_j}a_{k}b_{k},
\end{equation*}
which together with \eqref{geq:Anjbnj} and \eqref{AnAN} yields that
\begin{align*}
\sum_{k=N'+1}^{n_j}a_{k}b_{k}&\geq\frac{1}{ M_2}\sum_{k=N'+1}^{n_j}a_kb_{n_j}\\
&=\frac{1}{M_2}b_{n_j}(A_{n_j}-A_{N'})\\
&\geq\frac{1}{2M_2}A_{n_j}b_{n_j}\\
&\geq\frac{1}{4M_2}\delta
\end{align*}
for $j\geq J$. This contradicts \eqref{akbkleq} with $\varepsilon=\frac{1}{4M_2}\delta$, and hence $\delta=0$.

It follows from the non-negativity of $\{A_kb_k\}$ that $\liminf\limits_{k\to\infty}A_kb_k\geq0$, which together with  $\limsup\limits_{k\to\infty}A_kb_k=0$ yields $\lim\limits_{k\rightarrow\infty}A_kb_k=0$. Since $\lim\limits_{k\rightarrow\infty}A_k=+\infty$, we have $\lim\limits_{k\rightarrow\infty}b_k=0$. As a result, we can get $b_k=o\left(\frac{1}{A_k}\right)$.
\end{proof}

We next establish a lemma and a proposition that will serve as important tools in the analysis of convergence rate.

\begin{lemma}\label{T2TtheOth}
Let $T$ be an $\alpha$-averaged nonexpansive operator, $\alpha\in(0,1)$, and let $\{\bmx^k\}$ be a sequence generated by \eqref{eq:IKMAiter}, where $\{\theta_k\}$ is defined by \eqref{IKMAthetak}. Then
$$
\langle T\bmx^k-\bmx^k,T^2\bmx^k-T^{\theta_{k}}\bmx^k\rangle=0,\ \ k\in\bbN.
$$
\end{lemma}
\begin{proof}
By the definition of $T^{\theta_{k}}$ in \eqref{def:Tthetak}, we have
\begin{align*}
T^2\bmx^k-T^{\theta_{k}}\bmx^k=T^2\bmx^k-T\bmx^k-\theta_k(T\bmx^k-\bmx^k).
\end{align*}
Then
$$
\langle T\bmx^k-\bmx^k,T^2\bmx^k-T^{\theta_{k}}\bmx^k\rangle=\langle T\bmx^k-\bmx^k,T^2\bmx^k-T\bmx^k\rangle-\theta_k\|T\bmx^k-\bmx^k\|^2.
$$
Substituting $\theta_k=\frac{-\langle T\bmx^k-\bmx^k,T\bmx^k-T^2\bmx^k\rangle}{\|T\bmx^k-\bmx^k\|^2}$ into the above equality, we immediately obtain the desired result.
\end{proof}

We remark that Lemma \ref{T2TtheOth} indicates the orthogonality of $T\bmx^k-\bmx^k$ and $T^2\bmx^k-T^{\theta_{k}}\bmx^k$. In addition to this property, to establish the convergence rate of TKMA, we still need the following lemma.
\begin{lemma}\label{Tthetamoto}
Let $T$ be an $\alpha$-averaged nonexpansive operator, $\alpha\in(0,1/2]$, and let $\{\bmx^k\}$ be a sequence generated by \eqref{eq:IKMAiter}, where
$\{\theta_k\}$ is defined by \eqref{IKMAthetak}. Then
$$
\left\|T^{\theta_{k+1}}\bmx^{k+1}-T^{\theta_{k}}\bmx^{k}\right\|\leq(\tau_k+|1-\tau_k|)\left\|\bmx^{k+1}-\bmx^k\right\|,\ \ k\in\bbN,
$$
where $\tau_k:=\frac{1+\theta_{k+1}}{1+\theta_{k}}$.
\end{lemma}
\begin{proof}
By the definition of $T^{\theta_{k}}$ in \eqref{def:Tthetak}, we have
$$
T^{\theta_{k}}\bmx^k-\bmx^k=(1+\theta_k)(T\bmx^k-\bmx^k),
$$
which together with Lemma \ref{T2TtheOth} implies that
$$
\langle T^2\bmx^k-T^{\theta_{k}}\bmx^k,T^{\theta_{k}}\bmx^k-\bmx^k\rangle=(1+\theta_k)\langle T^2\bmx^k-T^{\theta_{k}}\bmx^k,T\bmx^k-\bmx^k\rangle=0.
$$
According to the definition of $\bmx^{k+1}$ in \eqref{eq:IKMAiter}, we then obtain that
\begin{align*}
\left\|\bmx^{k+1}-\bmx^k\right\|^2&=\left\|(1-t)(T^2\bmx^k-T^{\theta_{k}}\bmx^k)+(T^{\theta_{k}}\bmx^k-\bmx^k)\right\|^2\\
&=(1-t)^2\left\|T^2\bmx^k-T^{\theta_{k}}\bmx^k\right\|^2+\left\|T^{\theta_{k}}\bmx^k-\bmx^k\right\|^2,
\end{align*}
which implies that
\begin{align}\label{xk1Ttheta}
\left\|\bmx^{k+1}-T^{\theta_{k}}\bmx^k\right\|\notag
&=\left\|(1-t)T^2\bmx^k+tT^{\theta_{k}}\bmx^k-T^{\theta_{k}}\bmx^k\right\|\\\notag
&=(1-t)\left\|T^2\bmx^k-T^{\theta_{k}}\bmx^k\right\|\notag\\
&\leq\|\bmx^{k+1}-\bmx^k\|.
\end{align}
Since $\alpha\in(0,1/2]$, it follows from item $(ii)$ of Proposition \ref{prop:thetarange} that \begin{equation}\label{neq:nonexpanTthetak}
\left\|T^{\theta_{k}}\bmx^{k+1}-T^{\theta_{k}}\bmx^{k}\right\|\leq\|\bmx^{k+1}-\bmx^{k}\|.
\end{equation}
Since $\tau_k:=\frac{1+\theta_{k+1}}{1+\theta_{k}}>0$, we have
\begin{align*}
T^{\theta_{k+1}}&=(1+\theta_{k+1})T-\theta_{k+1}\mI\\
&=\tau_k(1+\theta_{k})T-\left(\tau_k(1+\theta_k)-1\right)\mI\\
&=\tau_k T^{\theta_{k}}+(1-\tau_k)\mI,
\end{align*}
which together with inequalities  \eqref{xk1Ttheta} and \eqref{neq:nonexpanTthetak} yields that
\begin{align*}
\left\|T^{\theta_{k+1}}\bmx^{k+1}-T^{\theta_{k}}\bmx^{k}\right\|&=\left\|\tau_k T^{\theta_{k}}\bmx^{k+1}+(1-\tau_k)\bmx^{k+1}-T^{\theta_{k}}\bmx^{k}\right\|\\
&=\left\|\tau_k (T^{\theta_{k}}\bmx^{k+1}-T^{\theta_{k}}\bmx^{k})+(1-\tau_k)(\bmx^{k+1}-T^{\theta_{k}}\bmx^{k})\right\|\\
&\leq\tau_k\left\|T^{\theta_{k}}\bmx^{k+1}-T^{\theta_{k}}\bmx^{k}\right\|+|1-\tau_k|\left\|\bmx^{k+1}-T^{\theta_{k}}\bmx^{k}\right\|\\
&\leq(\tau_k+|1-\tau_k|)\|\bmx^{k+1}-\bmx^k\|.
\end{align*}
This completes the proof.
\end{proof}

\begin{theorem}\label{convergerate}
Let $T$ be an $\alpha$-averaged nonexpansive operator, $\alpha\in(0,1/2]$, and let $\{\bmx^k\}$ be a sequence generated by \eqref{eq:IKMAiter}, where
$\{\theta_k\}$ is defined by \eqref{IKMAthetak}. If $\sum_{k=0}^{\infty} |\theta_{k+1}-\theta_k|<+\infty$, then
\begin{equation*}
\|\bmx^{k+1}-\bmx^k\|=o(\frac{1}{\sqrt{k}}).
\end{equation*}
\end{theorem}
\begin{proof}
We shall employ Lemma \ref{generalconvergerate} to prove this theorem. It follows from item $(i)$ of Proposition \ref{prop:thetarange} that $\theta_k\in[1-2\alpha,1]$, which together with $t\in(0,1)$ and $\alpha\in(0,1/2]$ implies that
\begin{equation}\label{tthetakscope}
1+t\theta_k>0, \ \ 2-t+t\theta_k\leq2.
\end{equation}
The averaged nonexpansiveness of $T$ implies its
nonexpansiveness, and hence
\begin{equation}\label{Tnonexpan}
\left\|T^2\bmx^k-T\bmx^k\right\|\leq\|T\bmx^k-\bmx^k\|.
\end{equation}
From the proof of Lemma \ref{lem:converge}, we can get
\begin{equation}\label{sumTxkbounded}
\sum_{k=0}^{\infty}\|T\bmx^{k}-\bmx^{k}\|^2<+\infty.
\end{equation}
According to the definition of $\bmx^{k+1}$ in \eqref{eq:IKMAiter}, and then using \eqref{tthetakscope} and \eqref{Tnonexpan}, we have
\begin{align*}
\left\|\bmx^{k+1}-\bmx^{k}\right\|&=\left\|(1-t)(T^2\bmx^k-T\bmx^k)+(1+t\theta_k)(T\bmx^k-\bmx^k)\right\|\\
&\leq(1-t)\left\|T^2\bmx^k-T\bmx^k\right\|+(1+t\theta_k)\|T\bmx^k-\bmx^k\|\\
&\leq2\|T\bmx^k-\bmx^k\|,
\end{align*}
 which together with \eqref{sumTxkbounded} implies
$$
\sum_{k=0}^{\infty}\|\bmx^{k+1}-\bmx^{k}\|^2<+\infty.
$$
Now we let $a_k:=1$ and $b_k:=\|\bmx^{k+1}-\bmx^{k}\|^2$ in Lemma \ref{generalconvergerate}. Then $\sum_{k=0}^{\infty}a_k b_k<+\infty$ and $\sum_{k=0}^{\infty}a_k=+\infty$.

Since $T$ is nonexpansive, it is obvious that $T^2$ is also nonexpansive, and hence
\begin{equation}\label{T2nonexpan}
\left\|T^2\bmx^{k+1}-T^2\bmx^k\right\|\leq\|\bmx^{k+1}-\bmx^k\|.
\end{equation}
According to the definition of $\bmx^{k+1}$ in \eqref{eq:IKMAiter}, and using \eqref{T2nonexpan} and Lemma \ref{Tthetamoto}, we obtain that
\begin{align*}
\left\|\bmx^{k+2}-\bmx^{k+1}\right\|&=\left\|(1-t)(T^2\bmx^{k+1}-T^2\bmx^{k})+t(T^{\theta_{k+1}}\bmx^{k+1}-T^{\theta_{k}}\bmx^{k})\right\|\\
&\leq(1-t)\left\|\bmx^{k+1}-\bmx^{k}\right\|+t(\tau_k+|1-\tau_k|)\left\|\bmx^{k+1}-\bmx^{k}\right\|.
\end{align*}
Let $d_k:=\frac{2t|\theta_{k+1}-\theta_k|}{1+\theta_k}$. It is easy to see from the definition of $\tau_k$ that
$$
(1-t)+t(\tau_k+|1-\tau_k|)\leq1+d_k.
$$
Then
\begin{equation}\label{neq:xk1mxkapprdecr}
\left\|\bmx^{k+2}-\bmx^{k+1}\right\|\leq (1+d_k)\left\|\bmx^{k+1}-\bmx^{k}\right\|.
\end{equation}
The fact $\sum_{k=0}^{\infty} |\theta_{k+1}-\theta_k|<+\infty$ implies that
\begin{equation}\label{limthetakplus1}
\lim _{k\rightarrow\infty}|\theta_{k+1}-\theta_k|=0.
\end{equation}
By item $(i)$ of Proposition \ref{prop:thetarange}, we have
\begin{equation*}
t|\theta_{k+1}-\theta_k|\leq d_k\leq\frac{t|\theta_{k+1}-\theta_k|}{1-\alpha}.
\end{equation*}
Then according to \eqref{limthetakplus1}, and using the squeeze theorem, we see that $\lim_{k\to\infty}d_k=0$. Let $c_k:=\left(1+d_k\right)^2$ in Lemma \ref{generalconvergerate}. Of course,
$$
\lim _{k\rightarrow\infty}c_k=1.
$$
Moreover, the inequality \eqref{neq:xk1mxkapprdecr} can be rewritten as
\begin{align*}
b_{k+1}\leq c_kb_{k}.
\end{align*}
Since $d_k\geq0$, we have $\ln c_k\geq0$. To prove that $\left|\sum_{k=0}^{\infty} \ln c_k\right|$ converges, it suffices to show the convergence of $\sum_{k=0}^{\infty} \ln c_k$. We first note that if $\ln c_k=0$ for some $k$, the corresponding term contributes zero to the series $\sum_{k=0}^{\infty} \ln c_k$. Therefore, the convergence depends only on the terms for which $\ln c_k>0$. In this case, since $\lim_{k\rightarrow\infty}\frac{\ln c_k}{2d_k+d_k^2}=1$, according to Theorem 8.21 of \cite{apostol1958mathematical}, if $\sum_{k=0}^{\infty} 2d_k+d_k^2$ converges, then $\sum_{k=0}^{\infty} \ln c_k$ also converges. The fact $\sum_{k=0}^{\infty} |\theta_{k+1}-\theta_k|<+\infty$ implies that $\sum_{k=0}^{\infty}d_k<+\infty$. Then it follows from Exercise 8.25 (a) of \cite{apostol1958mathematical} that $\sum_{k=0}^{\infty}d_k^2<+\infty$, and hence $\sum_{k=0}^{\infty} 2d_k+d_k^2$ converges. Therefore, we have $|\sum_{k=0}^{\infty} \ln c_k|<+\infty$. As a result, by Lemma \ref{generalconvergerate}, we have $\|\bmx^{k+1}-\bmx^{k}\|^2=o(\frac{1}{k})$. This completes the proof.
\end{proof}

\section{Numerical results}
In this section, we demonstrate the performance of the proposed algorithm through two numerical examples. Specifically, we compare the TKMA with the FPPA, the Proximal Gradient Algorithm (PGA), the Fast KM algorithm \eqref{FastKM} and the Halpern algorithm \eqref{tradHalpern}. In section \ref{ImageDenoising}, we address the TV regularized image denoising problem \eqref{TVdenoise}. Then, in section \ref{lowrank}, we consider a low-rank matrix completion problem, which frequently arises in applications such as collaborative filtering and image inpainting.

All experiments were conducted in MATLAB on a desktop running Windows 11 (64-bit), equipped with an Intel Core i9-14900K processor (3.20 GHz), 32 GB of DDR5 RAM, and a NVIDIA GeForce RTX 5090. Below, we present the figure-of-merits employed to evaluate and compare the performance of the algorithms. For the image denoising problem, we consider three metrics: the objective function value (OFV), the peak signal-to-noise ratio (PSNR) and the fixed-point (FP) residual $\|T\bmx^k-\bmx^{k}\|$. The PSNR is defined as
$$
\text{PSNR}(\bmx^k):=10\cdot\log_{10}\left(\frac{\text{MAX}^2}{\text{MSE}(\bmx^k)}\right),
$$
where MAX denotes the maximum possible pixel value of the image. This metric evaluates the quality of a reconstructed image compared to the ground truth. For the low-rank matrix completion problem, in addition to the OFV and the FP residual, we also employ the relative error (RE), defined as
$$
\text{RE}(\bmX^k):=\frac{\|\bmX^k-\bmX^{k-1}\|_F}{\|\bmX^k\|_F}.
$$
A steadily diminishing RE indicates the progressive convergence of the algorithm.

\subsection{Image Denoising Problem}\label{ImageDenoising}
In this subsection, we consider the TV regularized image denoising problem defined in \eqref{TVdenoise}. The first-order difference matrix $\bmB\in\bbR^{2d\times d}$ in the regularization term $\phi(\bmB\bmu)$ is defined by
$
\bmB:=\left[\begin{array}{c}
\bmI_N\otimes \bmD \\
\bmD\otimes \bmI_N
\end{array}\right],
$
where $\otimes$ denotes the Kronecker product, $N=\sqrt{d}$, $\bmI_N$ is the $N\times N$ identity matrix, and $\bmD$ is the $N\times N$ backward difference matrix with $\bmD_{j,j}=1$, $\bmD_{j,j-1}=-1$ for $j=2,3,\ldots,N$, and all other entries being zero.

As discussed in Section \ref{FFO}, the optimization problem \eqref{TVdenoise} can be solved by applying the fixed-point iteration of operator $T_2$ defined in \eqref{operatorTfordenoising}, namely,\vspace{-0.5em}
\begin{equation*}
\bmy^k=T_2 \bmy^{k-1}.
\end{equation*}
\vspace{-0.5em}The complete FPPA for solving problem \eqref{TVdenoise} is summarized as follows.

\hspace{-1.8em}\rule[0em]{13cm}{0.1em}\\\vspace{-0.5em}
\text{{\bf FPPA for TV denoising}}\\
\hspace{-1.5em}\rule[0em]{13cm}{0.1em}\\
{\bf Input:} Given the noisy image $\bmx$; the\ regularization parameter $\mu>0$; $\lambda<\frac{1}{4}$; the number of iterations $K$.

\hspace{-1.85em}{\bf Initialization:} $\bmy^0=\bmB\bmx$ and $k=1$.

\hspace{-1.85em}{\bf repeat}

$\bmy^k=\left(\mI-\prox_{\frac{1}{\lambda}\phi}\right)\left(\bmB\bmx+(\bmI_{2d}-\lambda \bmB\bmB^{\top})\bmy^{k-1}\right)$

$k=k+1$

\hspace{-1.85em}{\bf until} $k>K$

\hspace{-1.85em}{\bf Output}: The denoised image $\bmu=\bmx-\lambda \bmB^{\top}\bmy^K$.

\vspace{-0.5em}

\hspace{-1.8em}\rule[0em]{13cm}{0.1em}

By Proposition 12.28 of \cite{bauschke2017convex}, $\mI-\prox_{\frac{1}{\lambda}\phi}$ is firmly nonexpansive, which is $1/2$-averaged nonexpansive from Remark 4.34 of \cite{bauschke2017convex}. Thus operator $T_2$ defined by \eqref{operatorTfordenoising} is averaged nonexpansive when $\lambda<\frac{1}{4}$. According to \eqref{eq:IKMAiter}, the complete iterative scheme of the TKMA for solving model \eqref{TVdenoise} is given as follows.

\hspace{-1.8em}\rule[0em]{13cm}{0.1em}\\\vspace{-0.5em}
\text{{\bf TKMA for TV denoising}}\\
\hspace{-1em}\rule[0em]{13cm}{0.1em}\\
{\bf Input:} Given the noisy image $\bmx$; the\ regularization parameter $\mu>0$; $\lambda<\frac{1}{4}$; the combination coefficient $t$; the number of iterations $K$.

\hspace{-1.85em}{\bf Initialization:} $\bmy^0=\bmB\bmx$ and $k=1$.

\hspace{-1.85em}{\bf repeat}

$\tilde{\bmy}^k = \left(\mI-\prox_{\frac{1}{\lambda}\varphi}\right)\left(\bmB\bmx+(\bmI_{2d}-\lambda \bmB\bmB^{\top})\bmy^{k-1}\right)$

$\bmz^k = \left(\mI-\prox_{\frac{1}{\lambda}\varphi}\right)\left(\bmB\bmx+(\bmI_{2d}-\lambda \bmB\bmB^{\top})\tilde{\bmy}^k\right)$

$\theta_k = \frac{-\langle \tilde{\bmy}^k-\bmy^{k-1},\tilde{\bmy}^k-\bmz^k\rangle}{\|\tilde{\bmy}^k-\bmy^{k-1}\|^2}$

$\bmv^k=(1+\theta_k)\tilde{\bmy}^k-\theta_k \bmy^{k-1}$

$\bmy^k = (1-t)\bmz^k+t \bmv^k$

$k=k+1$

\hspace{-1.85em}{\bf until} $k>K$

\hspace{-1.85em}{\bf Output}: The denoised image $\bmu=\bmx-\lambda \bmB^{\top}\bmy^K$.

\vspace{-0.4em}
\hspace{-1.8em}\rule[0em]{13cm}{0.1em}
\vspace{-1.2em}

In our experiments, the $256\times256$ images of `Cameraman' and `Lighthouse' serve as the original images  $\bmh$. The noisy images are modeled as\vspace{-1em}
\begin{equation*}
\bmx=\bmh+\bm\eta\vspace{-0.8em}
\end{equation*}
with $\bm\eta\thicksim N(\bm0,\sigma^2\bmI)$ represents Gaussian noise. We apply each algorithm to the noisy image of `Cameraman' and `Lighthouse' with noise level $\sigma=15$ and $\sigma=25$, respectively. For the `Cameraman' image, the parameters $\mu$ and $\lambda$ are set to $10$ and $0.999/4$, respectively. For the `Lighthouse' image, the parameters $\mu$ and $\lambda$ are selected as $17$ and $0.999/4$, respectively. In our TKMA, the parameter $t$ is chosen as $1/2$. For the FPPA and the Halpern algorithm (with $\lambda_k=\frac{1}{k+1}$ \cite{halpern1967fixed} or $\lambda_k=\frac{1}{\varphi_k+1}$ \cite{he2024convergence}), no additive algorithmic parameters need to be adjusted. For the Fast KM algorithm, the parameters $s$ and $\alpha$ are set to $1$ and $50$, respectively, based on our fine-tuning. In the competing algorithms, the primary computational cost is concentrated on the operation of operator $T_2$, while the number of operations of $T_2$ in each iteration varies across different algorithms. As an example, a single iteration in TKMA is composed of two operations of $T_2$. To ensure a fair comparison, denoised images are presented after an equal number of the operations of $T_2$ across all methods. We also evaluate performance by plotting PSNR and OFV against CPU time, which provides a practical measure of computational efficiency.

We first  evaluate the performance of each algorithm on denoising the noisy `Cameraman' image with noise level $\sigma=15$. In Figure \ref{denoisedfigure}, we show the original image, the noisy image, and the denoised images from each algorithm following 20 iterations of $T_2$. Figure \ref{sigma15PSNRandOFV} plots the PSNR and OFV against CPU time. These results demonstrate that TKMA outperforms the other algorithms, achieving higher PSNR values and lower objective function values in less time.

\begin{figure*}[htbp]
\centering
\begin{minipage}[c]{0.25\textwidth}
    \centering
    \includegraphics[width=\textwidth]{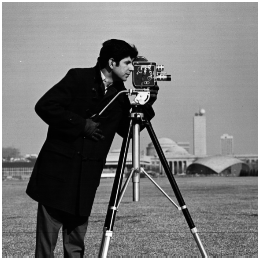}
    \\
    (a) Original image
\end{minipage}%
\hfill
\begin{minipage}[c]{0.7\textwidth}
    \centering
    \begin{minipage}[b]{\textwidth}
        \centering
        \includegraphics[width=0.32\textwidth]{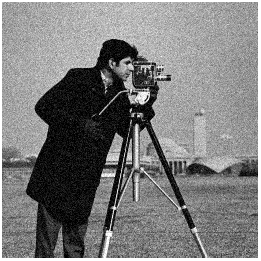}
        \includegraphics[width=0.32\textwidth]{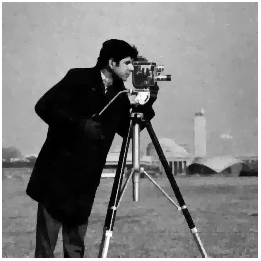}
        \includegraphics[width=0.32\textwidth]{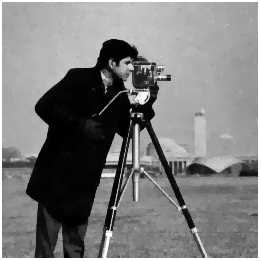}
        \\
        \hspace{-1.5em}(b) Noisy image  \ \  \quad\quad \ (c) TKMA \ \ \quad\quad\quad \  (d) FPPA  \quad
    \end{minipage}
    \\ %
    \vspace{0.2cm}
    \begin{minipage}[b]{\textwidth}
        \centering
        \includegraphics[width=0.32\textwidth]{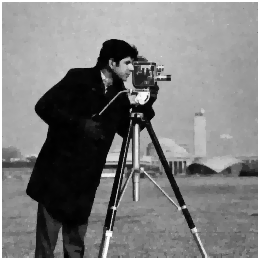}
        \includegraphics[width=0.32\textwidth]{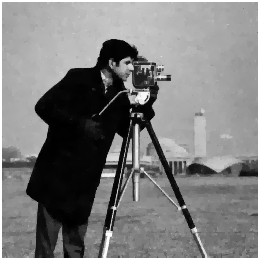}
        \includegraphics[width=0.32\textwidth]{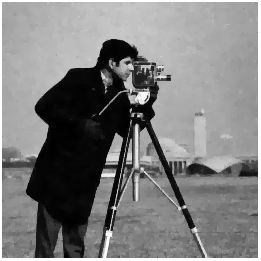}
        \\
        \parbox[t]{0.29\textwidth}{\centering (e) Halpern with\\ $\lambda_k=\frac{1}{k+1}$} \quad
        \parbox[t]{0.29\textwidth}{\centering (f) Halpern with\\ $\lambda_k=\frac{1}{\varphi_k+1}$} \quad
        \parbox[c]{0.29\textwidth}{\centering (g) Fast KM}
    \end{minipage}
\end{minipage}
\vspace{-0.7em}\caption{Comparison of the original image, noisy image, and denoised images obtained by the competing algorithms. (a) Original image of `Cameraman'; (b) Noisy image with noise level $\sigma=15$; (c) Denoised image using TKMA; (d) Denoised image using FPPA; (e) Denoised image using Halpern algorithm with $\lambda_k=\frac{1}{k+1}$; (f) Denoised image using Halpern algorithm with $\lambda_k=\frac{1}{\varphi_k+1}$; (g) Denoised image using Fast KM algorithm.}
\label{denoisedfigure}
\end{figure*}

\begin{figure}[!htbp]
\label{PSNRandOFV}
\centering
\begin{tabular}{cc}
\includegraphics[width=0.46\textwidth, height=0.23\textheight]{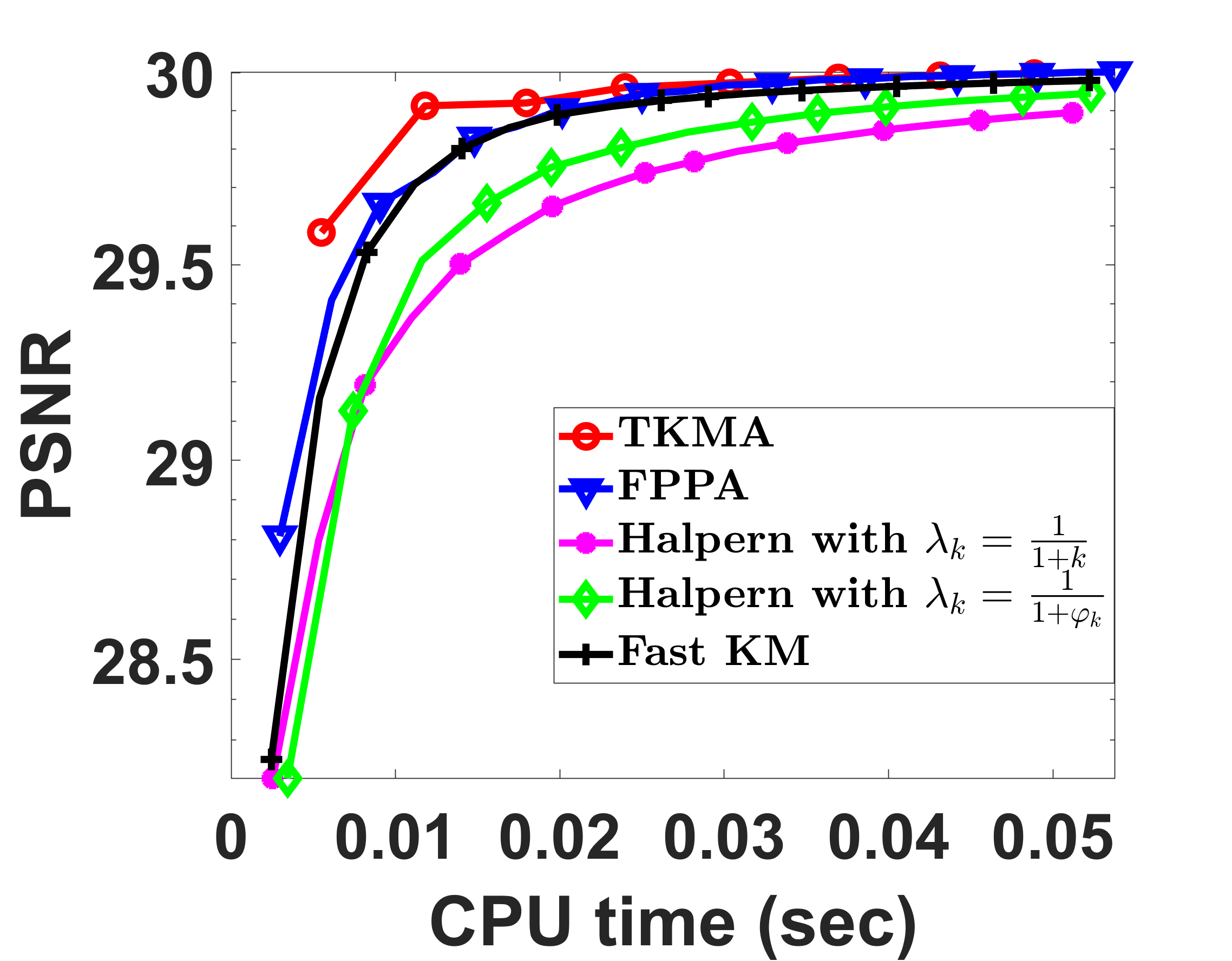}
\hspace{-1em}
&\includegraphics[width=0.47\textwidth, height=0.24\textheight]{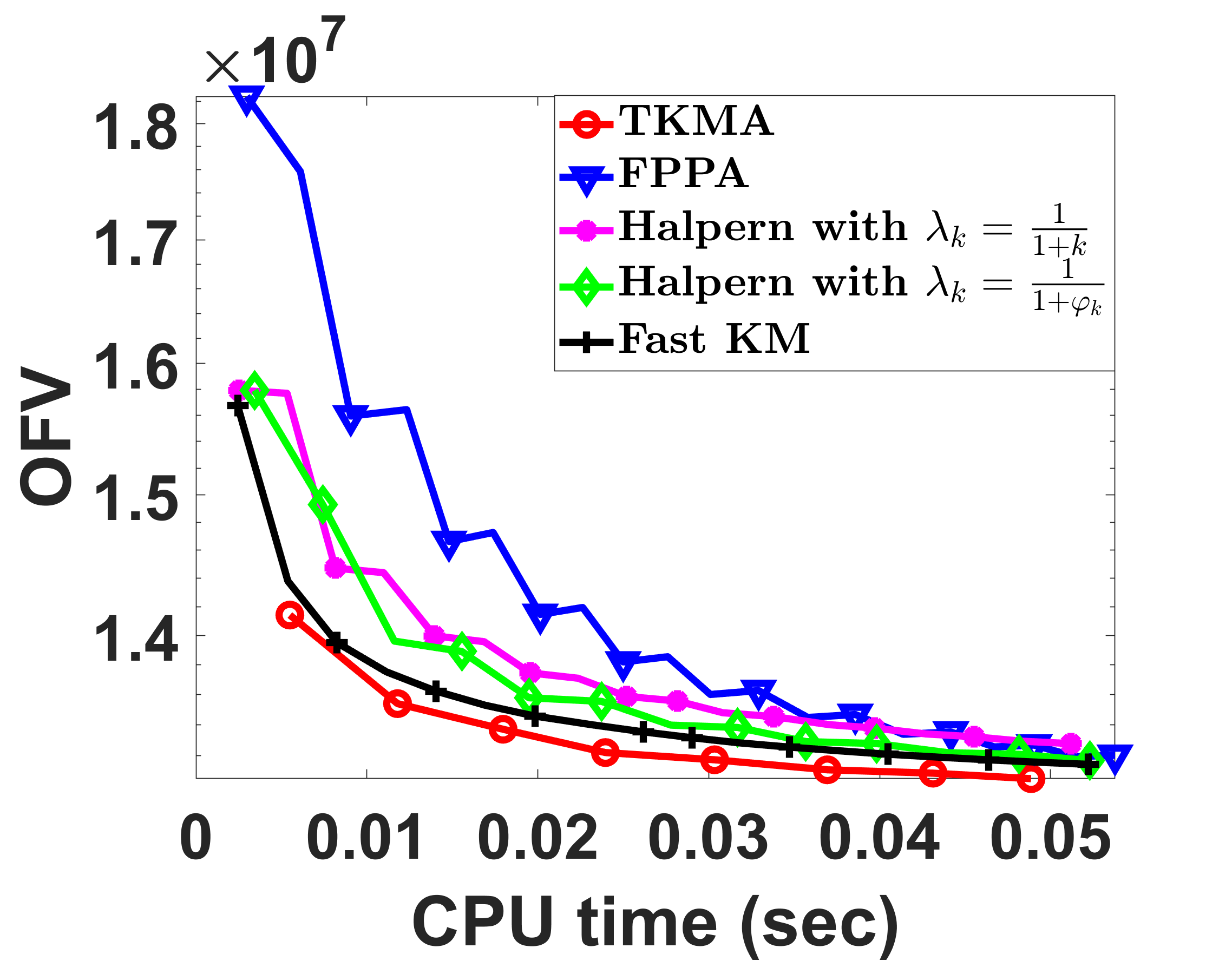}
\end{tabular}
\vspace{-1em}\caption{PSNR (left) and OFV (right) against CPU time for TKMA, FPPA, Halpern algorithm with $\lambda_k=\frac{1}{k+1}$, Halpern algorithm with $\lambda_k=\frac{1}{\varphi_k+1}$,
and Fast KM algorithm, under noise level $\sigma=15$.}
\label{sigma15PSNRandOFV}
\end{figure}

We then evaluate the performance of each algorithm for denoising the noisy `Lighthouse' image with a noise level of $\sigma=25$. Figure \ref{denoisedfigure2} shows the original image, the noisy image, and the denoised images obtained from each algorithm following 20 iterations of $T_2$. Figure \ref{sigma25PSNRandOFV} presents the PSNR and OFV against CPU time, which again confirms the superior performance of the TKMA. Note that at a CPU time of approximately 0.05 seconds in Figure \ref{sigma15PSNRandOFV}, the iteration counts for TKMA, FPPA, Halpern ($\lambda_k=\frac{1}{k+1}$), Halpern ($\lambda_k=\frac{1}{\varphi_k+1}$), and Fast KM are $9$, $22$, $20$, $14$, and $20$, respectively. Similar results are observed for the case $\sigma=25$ in Figure \ref{sigma25PSNRandOFV}. Notably, TKMA involves two operations of the operator $T$ per iteration, which incurs a higher per-iteration cost than the other methods.

\begin{figure*}[!htbp]
\centering
\begin{minipage}[c]{0.257\textwidth}
    \centering
    \includegraphics[width=\textwidth]{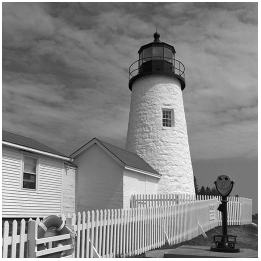}
    \\
    (a) Original image
\end{minipage}%
\hfill
\begin{minipage}[c]{0.7\textwidth}
    \centering
    \begin{minipage}[b]{\textwidth}
        \centering
        \includegraphics[width=0.32\textwidth]{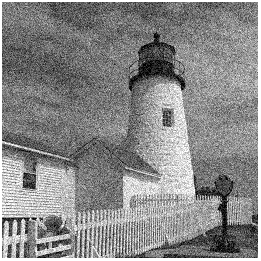}
        \includegraphics[width=0.32\textwidth]{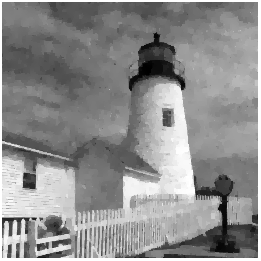}
        \includegraphics[width=0.32\textwidth]{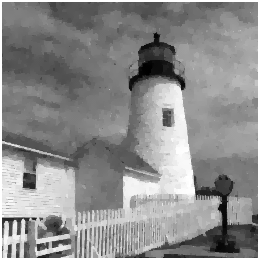}
        \\
        \hspace{-1.5em}(b) Noisy image  \ \  \quad\quad \ (c) TKMA \ \ \quad\quad\quad \  (d) FPPA \quad
    \end{minipage}
    \\
    \vspace{0.3cm}
    \begin{minipage}[b]{\textwidth}
        \centering
        \includegraphics[width=0.32\textwidth]{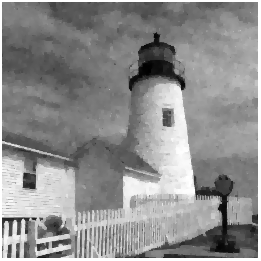}
        \includegraphics[width=0.32\textwidth]{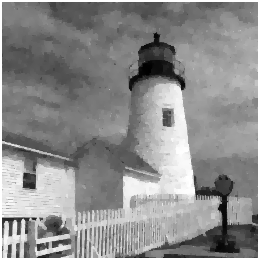}
        \includegraphics[width=0.32\textwidth]{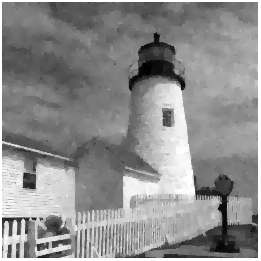}
        \\
        \parbox[t]{0.29\textwidth}{\centering (e) Halpern with\\ $\lambda_k=\frac{1}{k+1}$} \quad
        \parbox[t]{0.29\textwidth}{\centering (f) Halpern with\\ $\lambda_k=\frac{1}{\varphi_k+1}$} \quad
        \parbox[c]{0.29\textwidth}{\centering (g) Fast KM}
    \end{minipage}
\end{minipage}
\vspace{-0.5em}\caption{Comparison of the original image, noisy image, and denoised images obtained by the competing algorithms. (a) Original image of `Lighthouse'; (b) Noisy image with noise level $\sigma=25$; (c) Denoised image using TKMA; (d) Denoised image using FPPA; (e) Denoised image using Halpern algorithm with $\lambda_k=\frac{1}{k+1}$; (f) Denoised image using Halpern algorithm with $\lambda_k=\frac{1}{\varphi_k+1}$; (g) Denoised image using Fast KM algorithm.}
\label{denoisedfigure2}
\end{figure*}

\begin{figure}[!htbp]
\label{PSNRandOFV4}
\hspace{-1em}
\centering
\hspace{-1em}
\begin{tabular}{cc}
\includegraphics[width=0.46\textwidth, height=0.235\textheight]{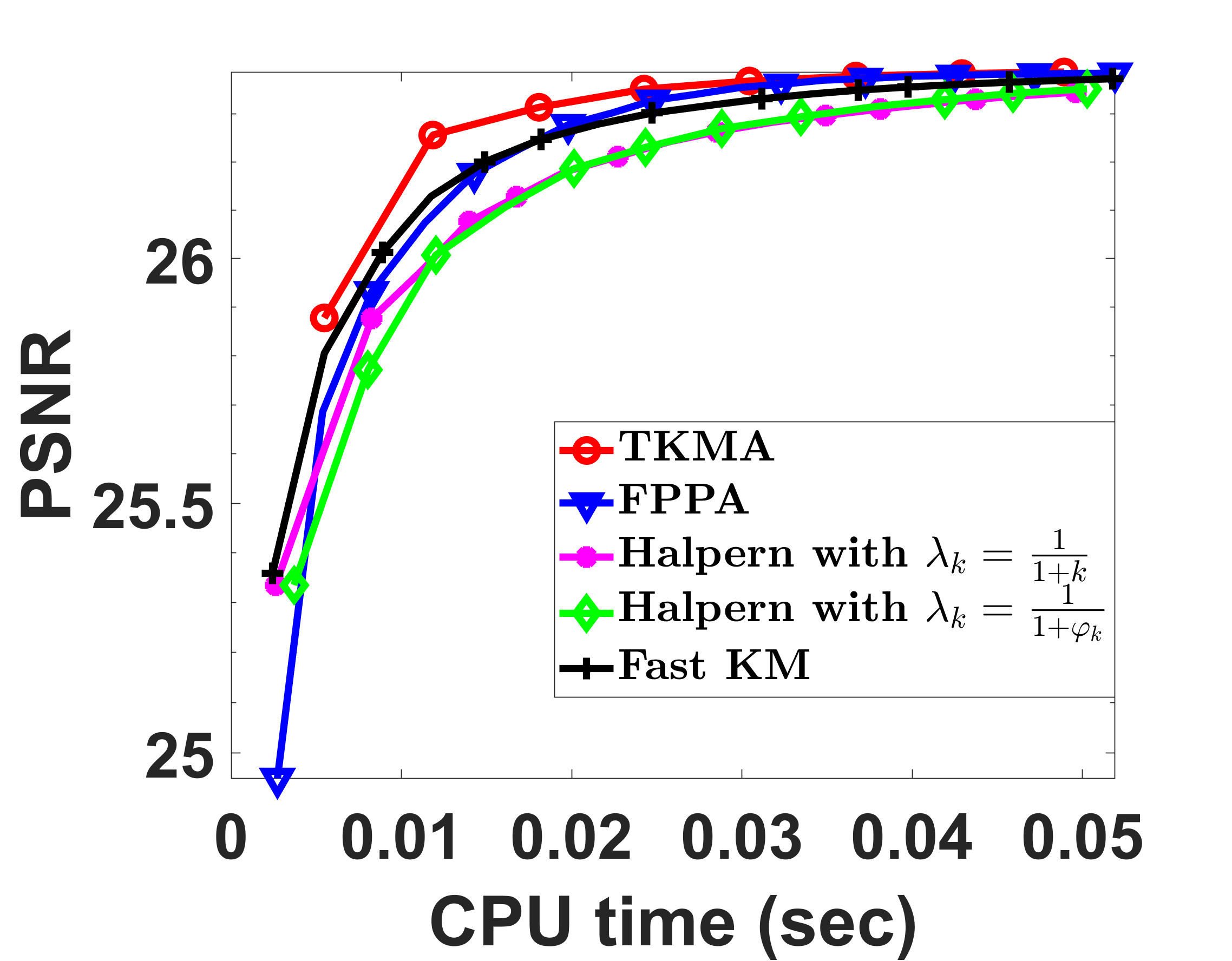}
\hspace{-1em}&
\includegraphics[width=0.46\textwidth, height=0.24\textheight]{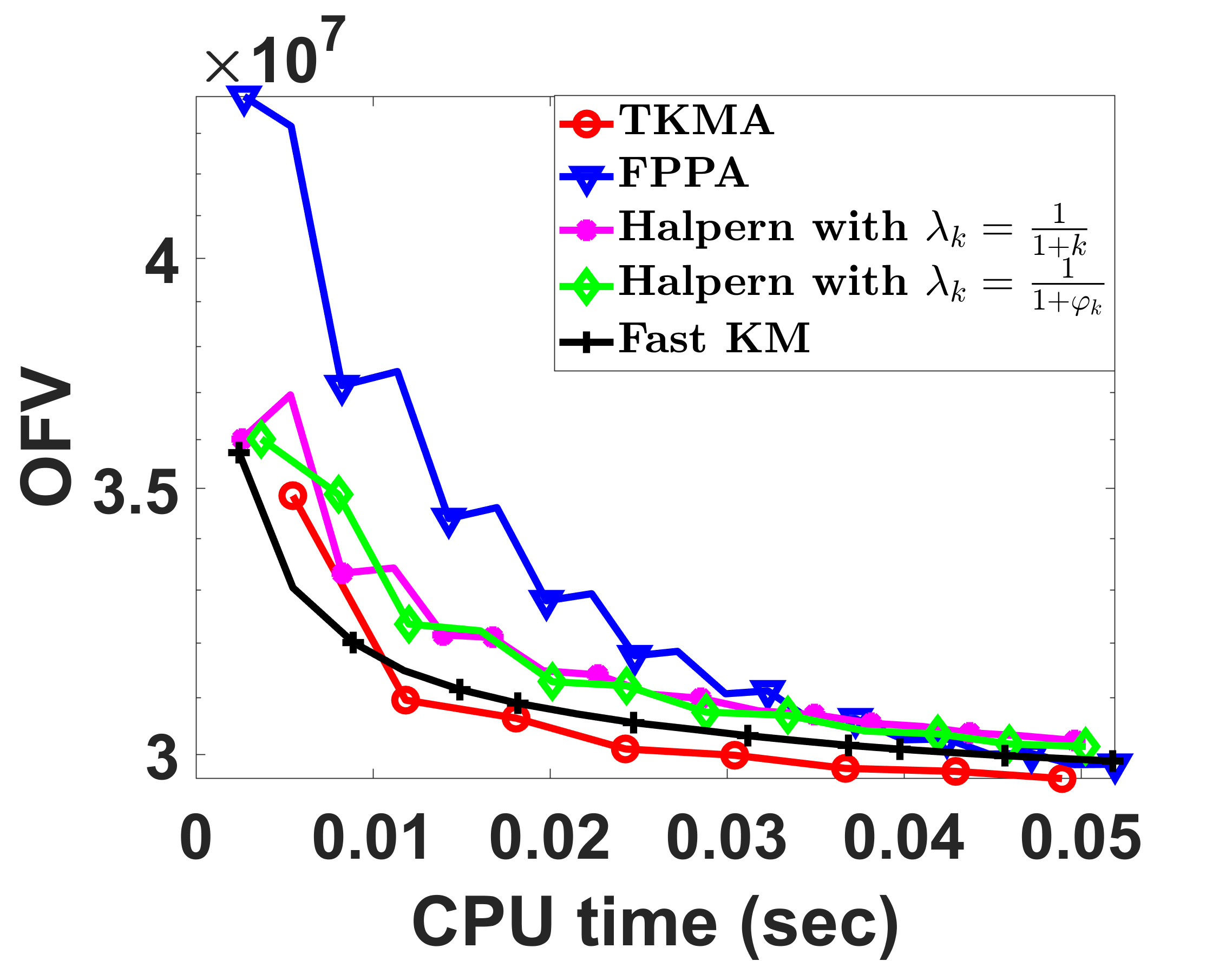}
\end{tabular}
\vspace{-0.5em}\caption{PSNR (left) and OFV (right) against CPU time for TKMA, FPPA, Halpern algorithm with $\lambda_k=\frac{1}{k+1}$, Halpern algorithm with $\lambda_k=\frac{1}{\varphi_k+1}$,
and Fast KM algorithm, under noise level $\sigma=25$.}
\label{sigma25PSNRandOFV}
\end{figure}

To further evaluate the convergence performance of the TKMA, we plot the FP residual against the iteration count in Figure \ref{Residualdenosing} for noise levels $\sigma = 15$ and $\sigma = 25$. The stopping criterion is $\|T\bmx^k-\bmx^k\|<10^{-10}$, with a maximum of $100,000$ iterations. Additionally, a reference line indicating the $1/k$ decay rate is superimposed for comparison. These results demonstrate that the residual of TKMA decays faster than the other compared methods and the $1/k$ benchmark. Specifically, under both noise levels, it requires fewer than one quarter of the iterations needed by the Fast KM algorithm to reach an accuracy of $10^{-10}$, and it also requires substantially fewer iterations than the Halpern algorithm with $\lambda_k=\frac{1}{k+1}$ and the $1/k$ reference line. Although TKMA requires twice the time per iteration compared to Fast KM, its performance on the FP residual still far surpasses that of Fast KM and the Halpern algorithm with $\lambda_k=\frac{1}{k+1}$.

\begin{figure}[!htbp]
\centering
\begin{tabular}{cc}
\hspace{-1em}
\includegraphics[width=0.47\textwidth]{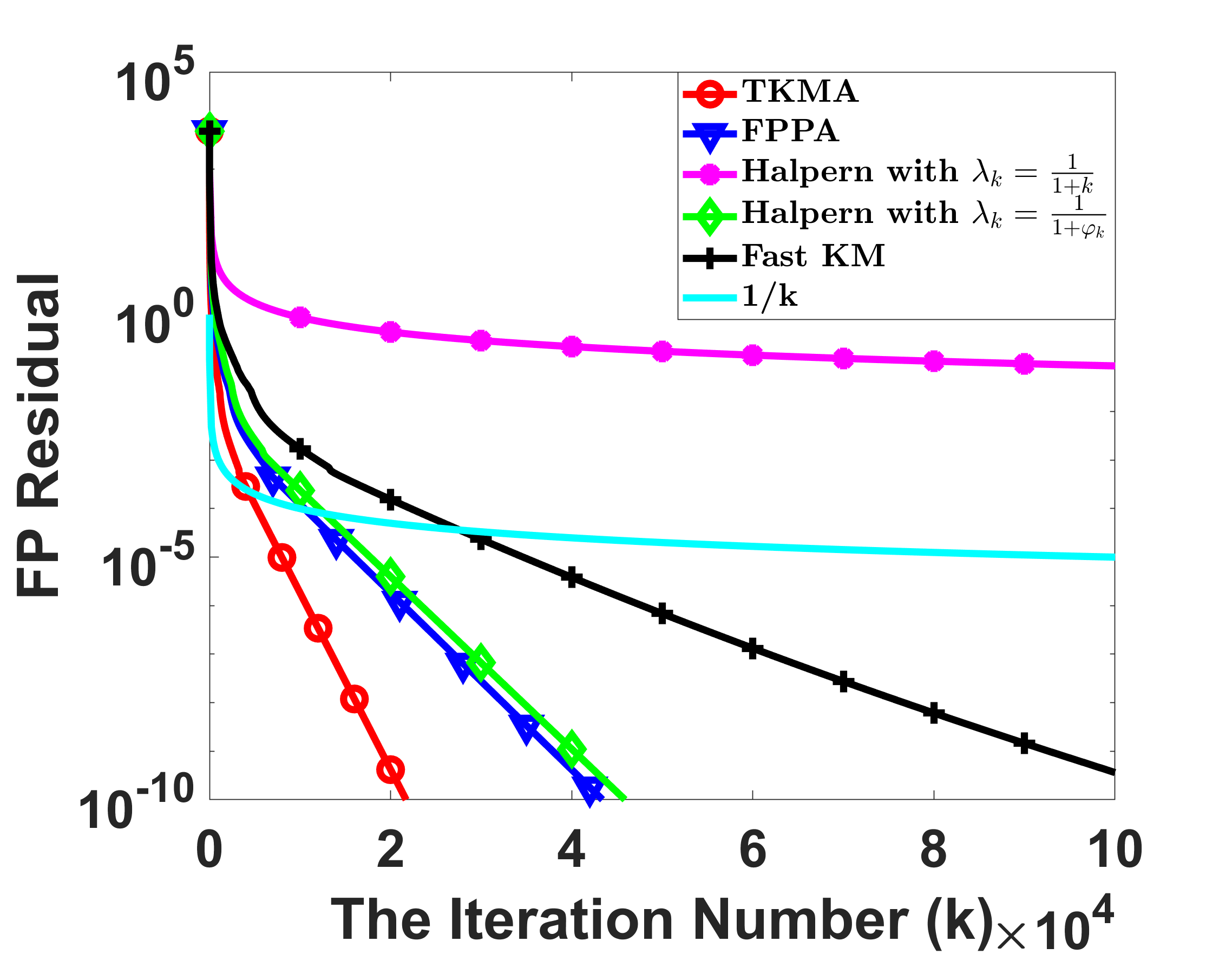}&
\includegraphics[width=0.47\textwidth]{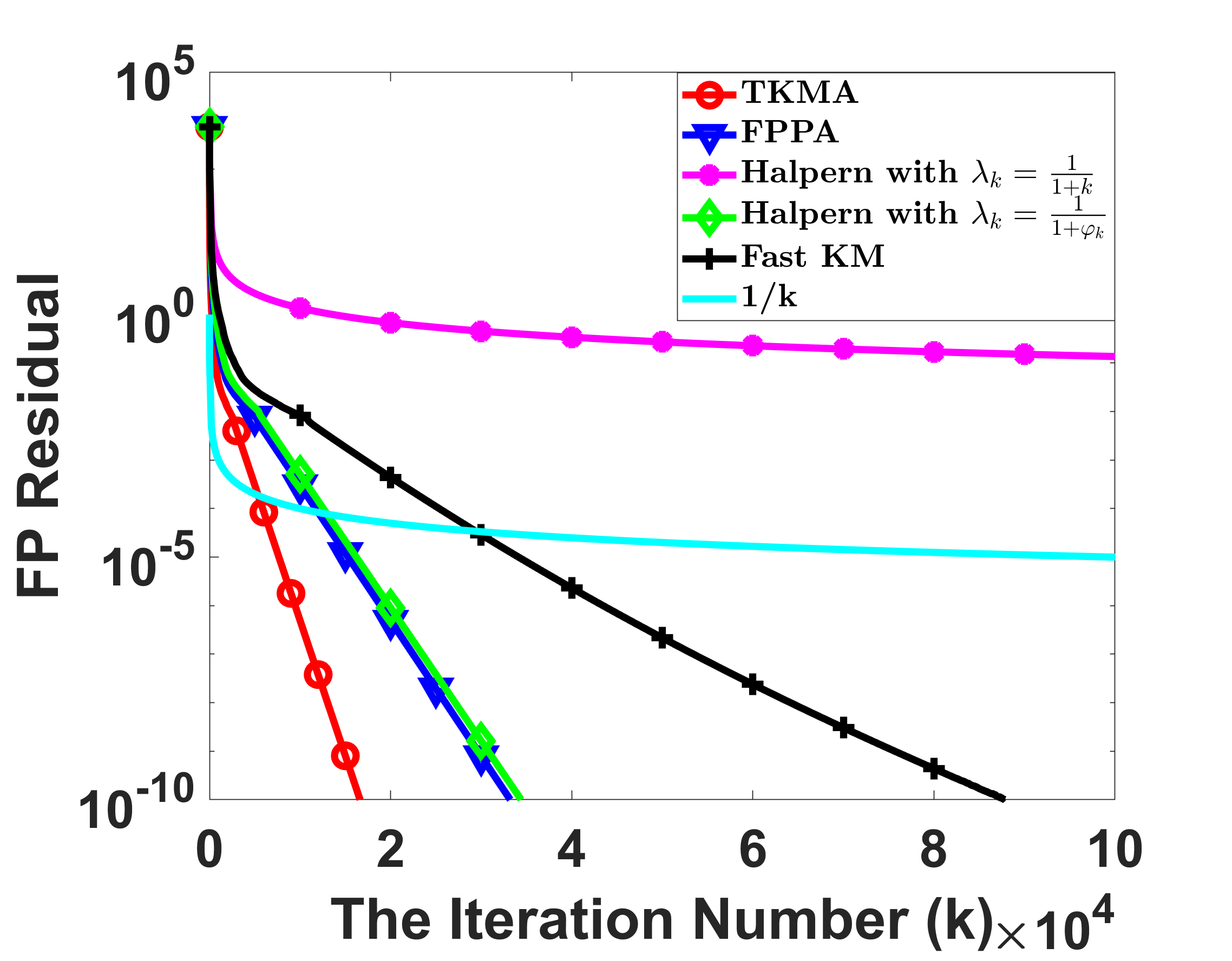}
\end{tabular}
\caption{FP residual versus the iteration number for TKMA, FPPA, Halpern algorithm with $\lambda_k=\frac{1}{k+1}$, Halpern algorithm with $\lambda_k=\frac{1}{\varphi_k+1}$,
and Fast KM algorithm, under noise levels $\sigma=15$ (left) and $\sigma=25$ (right).}
\label{Residualdenosing}
\end{figure}

Finally, we plot in Figure \ref{sumthetak} the partial sums $\sum_{k=0}^{n} |\theta_{k+1}-\theta_k|$ against the iteration count $n$. The observed convergence trend of the partial sums provides empirical evidence that the infinite series $\sum_{k=0}^{\infty} |\theta_{k+1}-\theta_k|$ may be bounded in the denoising scenario. This result aligns with the theoretical condition in Theorem \ref{convergerate} and offers experimental support for the convergence rate of TKMA established therein.

\begin{figure}[!htbp]
\centering
\begin{tabular}{cc}
\includegraphics[width=0.5\textwidth]{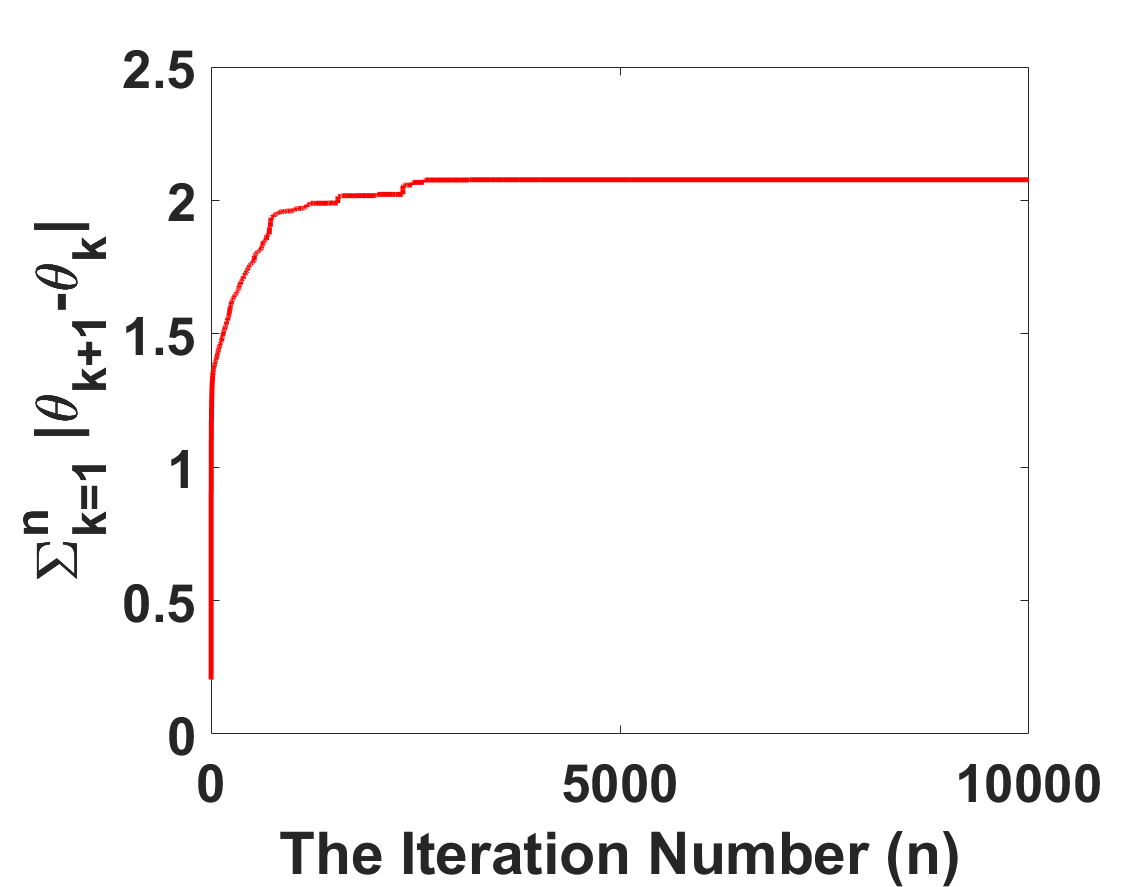}
\end{tabular}
\caption{$\sum_{k=0}^{n} |\theta_{k+1}-\theta_k|$ for TKMA against the iteration count $n$.}
\label{sumthetak}
\end{figure}

\subsection{Low-rank Matrix Completion Problem}\label{lowrank}
In this subsection, we consider a well-known low-rank matrix completion problem, which can be formulated as\vspace{-0.4em}
\begin{equation}\label{lowmatrix}\vspace{-0.4em}
\argmin_{\bmX\in\bbRnn} \frac{1}{2}\|P_{\Omega}(\bmX-\bmA)\|^2_F \ \ \text{subject to} \ \ \|\bmX\|_{\ast}\leq r,
\end{equation}
where $P_{\Omega}$ denotes the projection onto a randomly sampled index set $\Omega$, $r$ is the supposed maximum rank, and the nuclear norm constraint $\|\bmX\|_{\ast}\leq r$ is applied to enforce a low-rank solution. The most expensive computational step involved in solving this problem is the projection onto the nuclear norm ball $\mathcal{C}=\{\bmX:\|\bmX\|_{\ast}\leq r\}$, which requires a singular value decomposition (SVD). In our experiments, the ground-truth matrix $\bmA$ was generated as a low-rank product $\bmU\bmV^{\top}$, where $\bmU$ and $\bmV$ are n-by-r matrices with entries sampled from a normal distribution. The observation set $\Omega$ contained $20\%$ of all entries (i.e., $\frac{1}{5}n^2$), selected uniformly at random.

We define $f(\bmX):=\frac{1}{2}\|P_{\Omega}(\bmX-\bmA)\|^2_F$, and denote by $P_{\mathcal{C}}$ the projection operator onto the set $\mathcal{C}$. The optimization problem \eqref{lowmatrix} can be addressed using a PGA as follows.
\vspace{0.5em}

\hspace{-1.8em}\rule[0em]{13cm}{0.1em}\\\vspace{-0.5em}
\text{{\bf PGA for Low-rank Matrix Completion}}\\
\hspace{-1.5em}\rule[0em]{13cm}{0.1em}\\
{\bf Input:} Given the matrix $\bmA\in\mathbb{R}^{n \times n}$; the observed entries index set $\Omega$; the maximum rank $r > 0$; the step size $\alpha > 0$; the maximum number of iterations $K$.

\hspace{-1.85em}{\bf Initialization:} Initialize the matrix $\bmX^0$ as a zero matrix, then $\bmX^0_{ij} = \bmA_{ij}$ for $(i,j) \in \Omega$. Set $k=1$.

\hspace{-1.85em}{\bf repeat}

\vspace{0.15em}

$\bmX^{k} = P_{\mathcal{C}}\left(\bmX^{k-1} - \alpha\nabla f(\bmX^{k-1})\right)$

$k=k+1$

\vspace{0.15em}

\hspace{-1.85em}{\bf until} $k>K$

\vspace{0.15em}

\hspace{-1.85em}{\bf Output}: Completed matrix $\bmX^K$.
\vspace{-0.4em}

\hspace{-1.8em}\rule[0em]{13cm}{0.1em}
\vspace{0em}

Let $T_3:=P_{\mathcal{C}}\circ\left(I - \alpha\nabla f\right)$. It is easy to see that this operator is averaged nonexpansive when $\alpha\in(0,2/L)$, where $L$ is the Lipschitz constant of $\nabla f$. Based on \eqref{eq:IKMAiter}, the TKMA for solving model \eqref{lowmatrix} can be given as follows.

\hspace{-1.8em}\rule[0em]{13cm}{0.1em}\\\vspace{-0.5em}
\text{{\bf TKMA for Low-rank Matrix Completion}}\\
\hspace{-1.5em}\rule[0em]{13cm}{0.1em}\\
{\bf Input:} Given the matrix $\bmA\in\mathbb{R}^{n \times n}$; the observed entries index set $\Omega$; the maximum rank $r > 0$; the step size $\alpha > 0$; the
combination coefficient $t$; the maximum number of iterations $K$.\vspace{0.25em}

\hspace{-1.85em}{\bf Initialization:} Initialize the matrix $\bmX^0$ as a zero matrix, and set $\bmX^0_{ij} = \bmA_{ij}$ for $(i,j) \in \Omega$. Set $k=1$.

\hspace{-1.85em}{\bf repeat}

$\tilde{\bmX}^{k} = P_{\mathcal{C}}\left(\bmX^{k-1} - \alpha\nabla f(\bmX^{k-1})\right)$

$\bmY^{k} = P_{\mathcal{C}}\left(\tilde{\bmX}^{k} - \alpha\nabla f(\tilde{\bmX}^{k})\right)$

$\theta_k = \frac{-\langle \tilde{\bmX}^k-\bmX^{k-1},\tilde{\bmX}^k-\bmY^k\rangle_F}{\|\tilde{\bmX}^k-\bmX^{k-1}\|_F^2}$

$\bmZ^k=(1+\theta_k)\tilde{\bmX}^k-\theta_k \bmX^{k-1}$

$\bmX^k = (1-t)\bmY^k+t \bmZ^k$

$k=k+1$

\hspace{-1.85em}{\bf until} $k>K$

\hspace{-1.85em}{\bf Output}: Completed matrix $\bmX^K$.

\vspace{-0.4em}
\hspace{-1.8em}\rule[0em]{13cm}{0.1em}
\vspace{0em}

In this experiment, we consider a $100\times 100$ matrix and explore both the scenarios of rank $r=30$ and $r=50$. The step size parameter $\alpha$ is set to $1.99$. The algorithmic parameter for TKMA is consistent with that used in the denoising experiment in section \ref{ImageDenoising}. Both the PGA and the Halpern algorithm (with $\lambda_k=\frac{1}{k+1}$ or $\lambda_k=\frac{1}{\varphi_k+1}$) require no parameter tuning. For the Fast KM algorithm, the parameters $s$ and $\alpha$ are set to $1$ and $1000$, respectively, based on our fine-tuning. The results are shown in Figures \ref{OFVandREforlowmatrixr30}-\ref{residualr3050}, where we compare the OFV, RE and FP residual against CPU time. From the experimental results, it is clear that TKMA outperforms all other methods in this scenario. In addition, we plot the FP residual against the iteration count in Figure \ref{Residualmatrix} for ranks $r = 30$ and $r = 50$, with a tolerance of $10^{-4}$ and a maximum of 1,000 iterations. Compared to the $1/k$ reference line, TKMA exhibits a much more rapid decrease in residual and reaches a substantially lower level.

\vspace{1em}
\begin{figure}[htbp]
\label{OFVforlowmatrixr30}
\centering
\begin{tabular}{cc}
\hspace{-1em}
\includegraphics[width=0.475\textwidth, height=0.23\textheight]{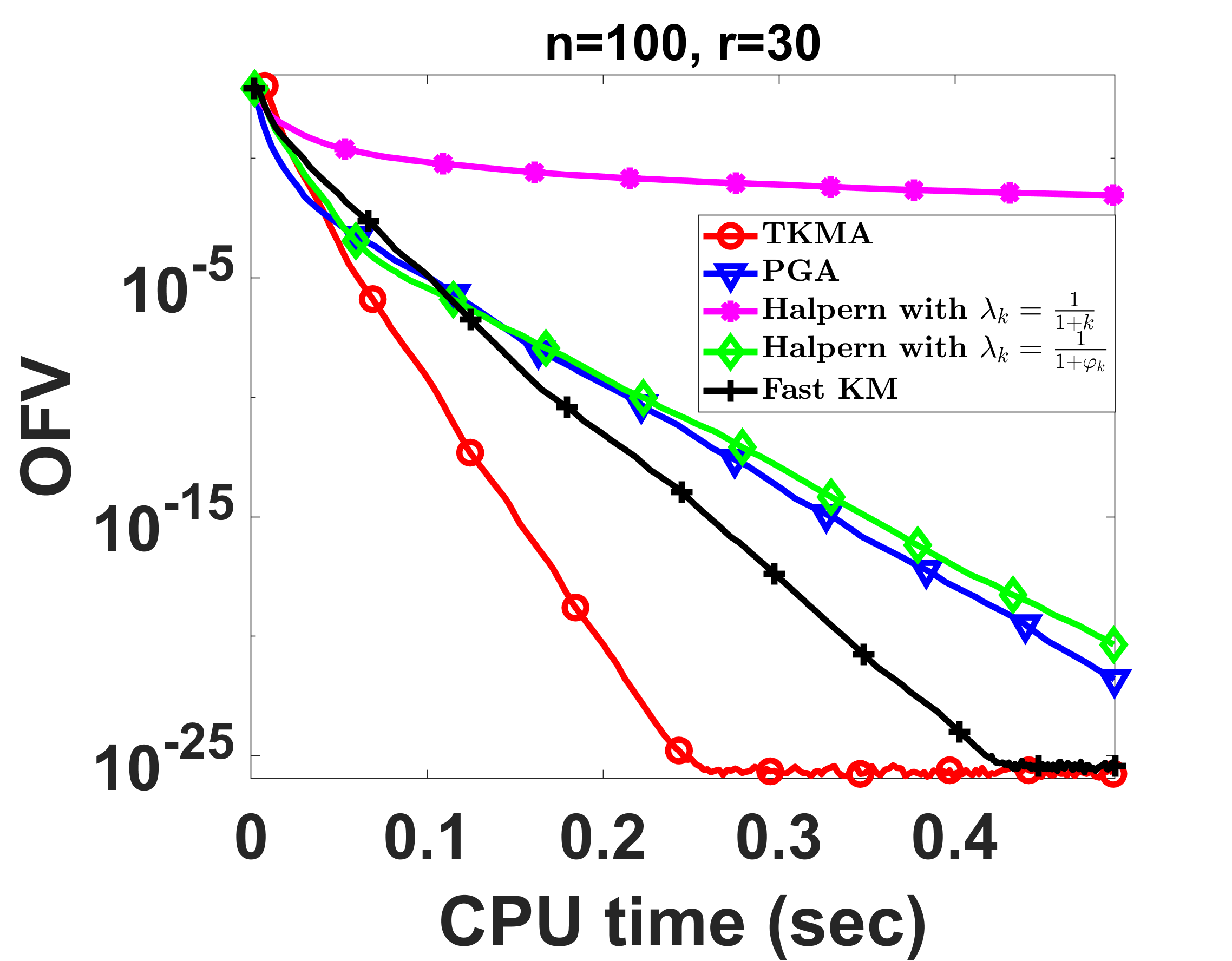}&
\includegraphics[width=0.475\textwidth, height=0.23\textheight]{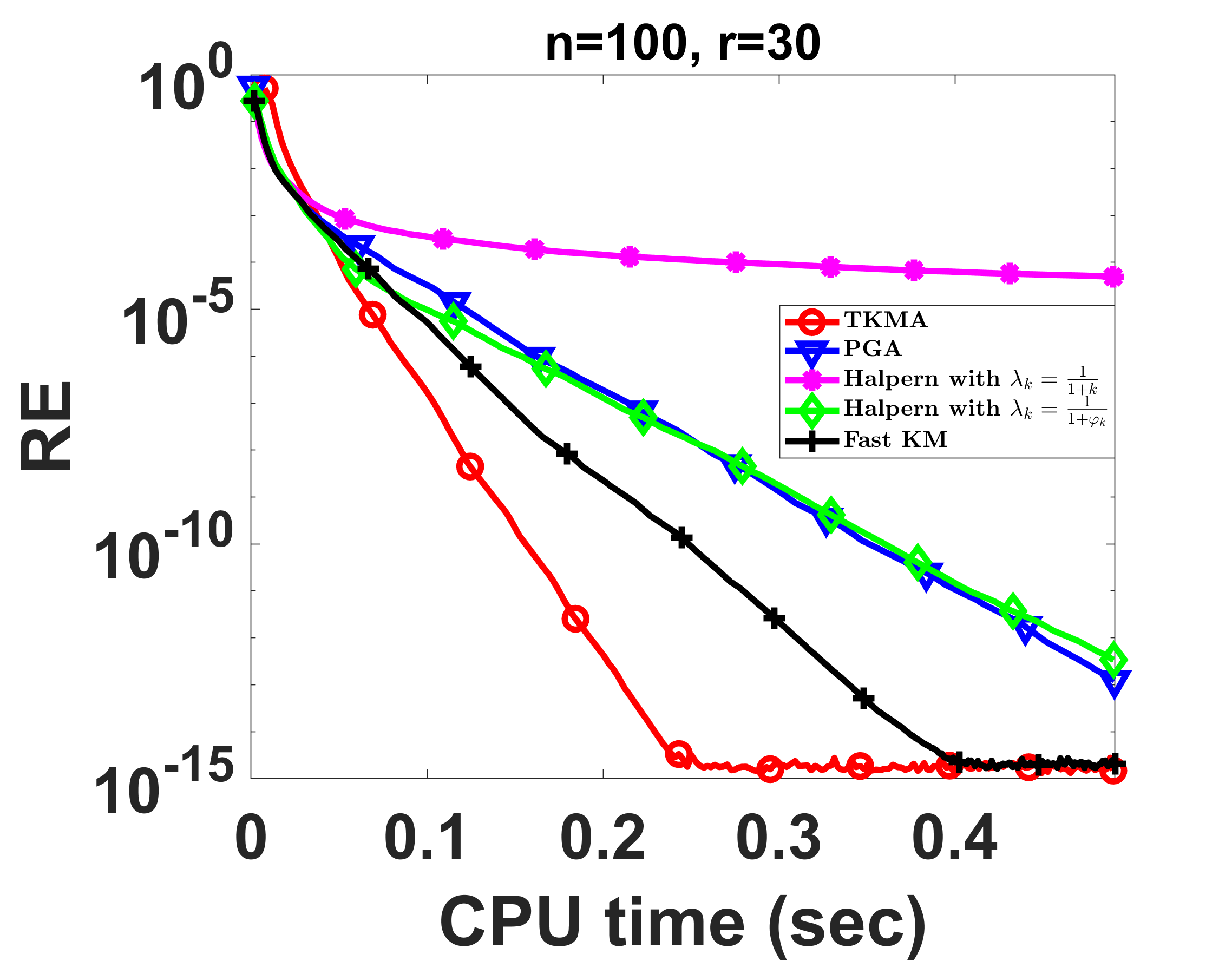}
\end{tabular}
\vspace{-0.7em}\caption{OFV (left) and RE (right) against CPU time for TKMA, PGA, Halpern algorithm with $\lambda_k=\frac{1}{k+1}$, Halpern algorithm with $\lambda_k=\frac{1}{\varphi_k+1}$,
and Fast KM algorithm with rank $r=30$.}
\label{OFVandREforlowmatrixr30}
\end{figure}

\begin{figure}[htbp]
\centering
\begin{tabular}{cc}
\hspace{-1em}
\includegraphics[width=0.475\textwidth, height=0.23\textheight]{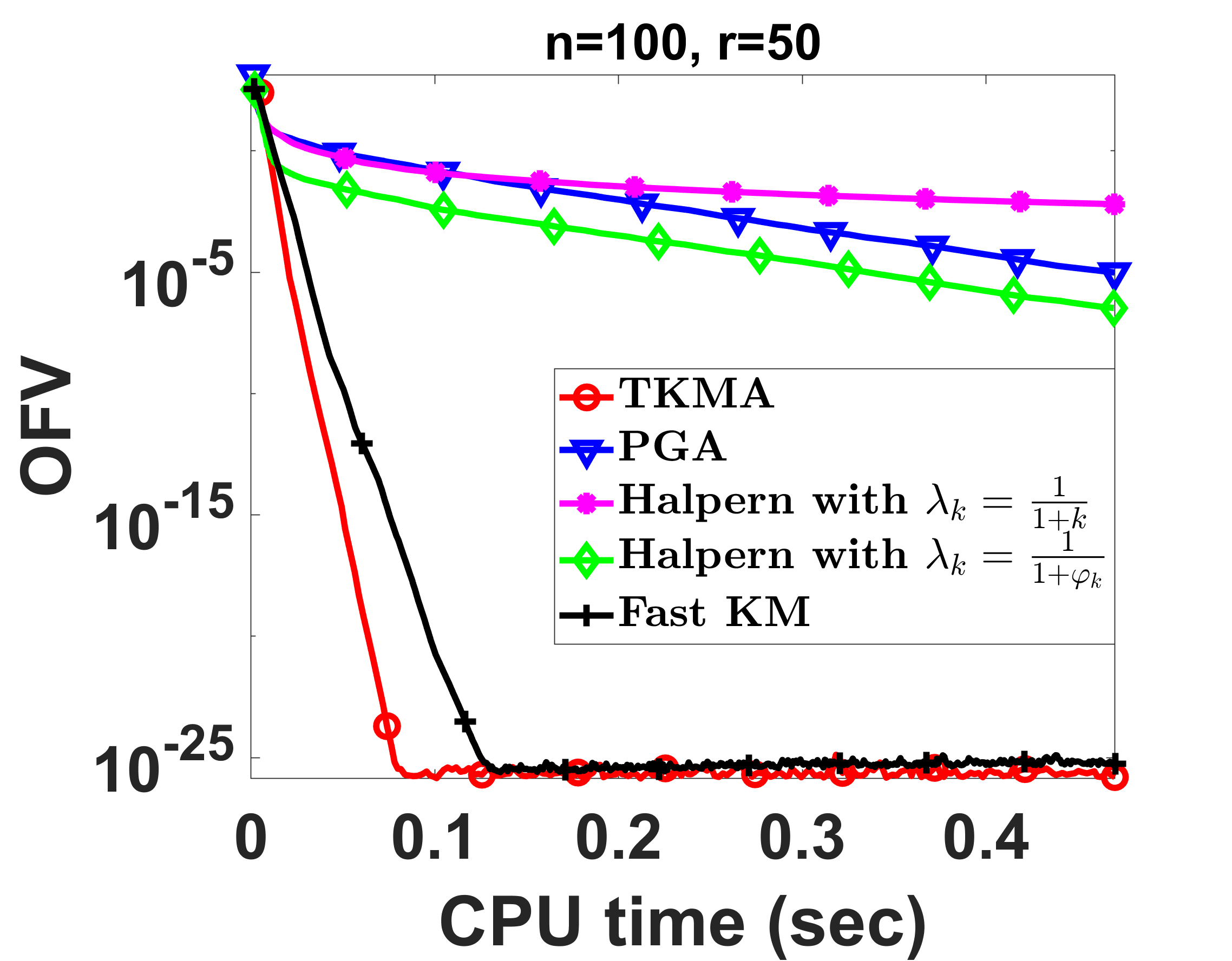}&
\includegraphics[width=0.475\textwidth, height=0.23\textheight]{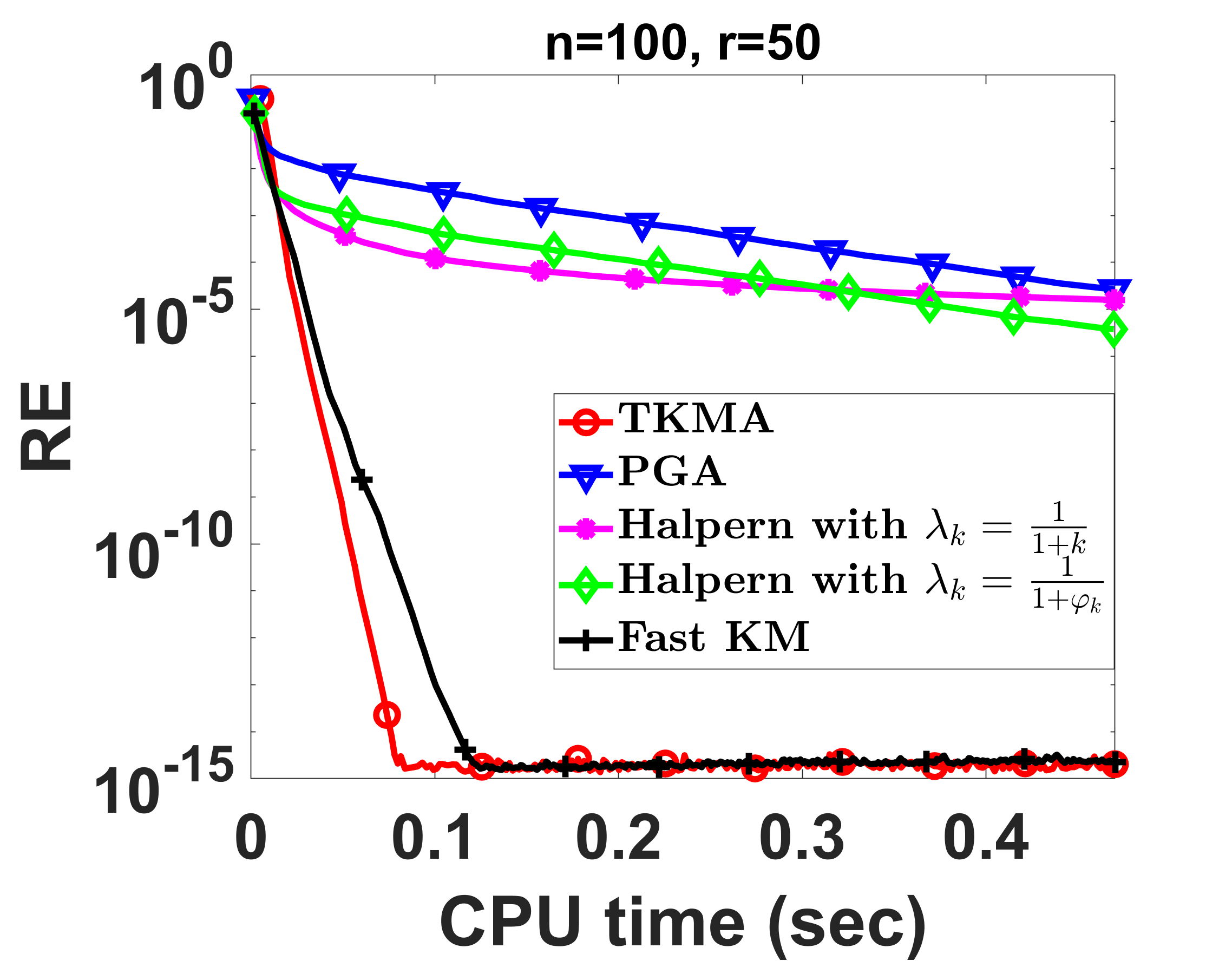}
\end{tabular}
\vspace{-0.7em}\caption{OFV (left) and RE (right) against CPU time for TKMA, PGA, Halpern algorithm with $\lambda_k=\frac{1}{k+1}$, Halpern algorithm with $\lambda_k=\frac{1}{\varphi_k+1}$,
and Fast KM algorithm with rank $r=50$.}
\label{OFVandREforlowmatrixr50}
\end{figure}

\begin{figure}[htbp]
\centering
\begin{tabular}{cc}
\hspace{-1em}
\includegraphics[width=0.475\textwidth, height=0.23\textheight]{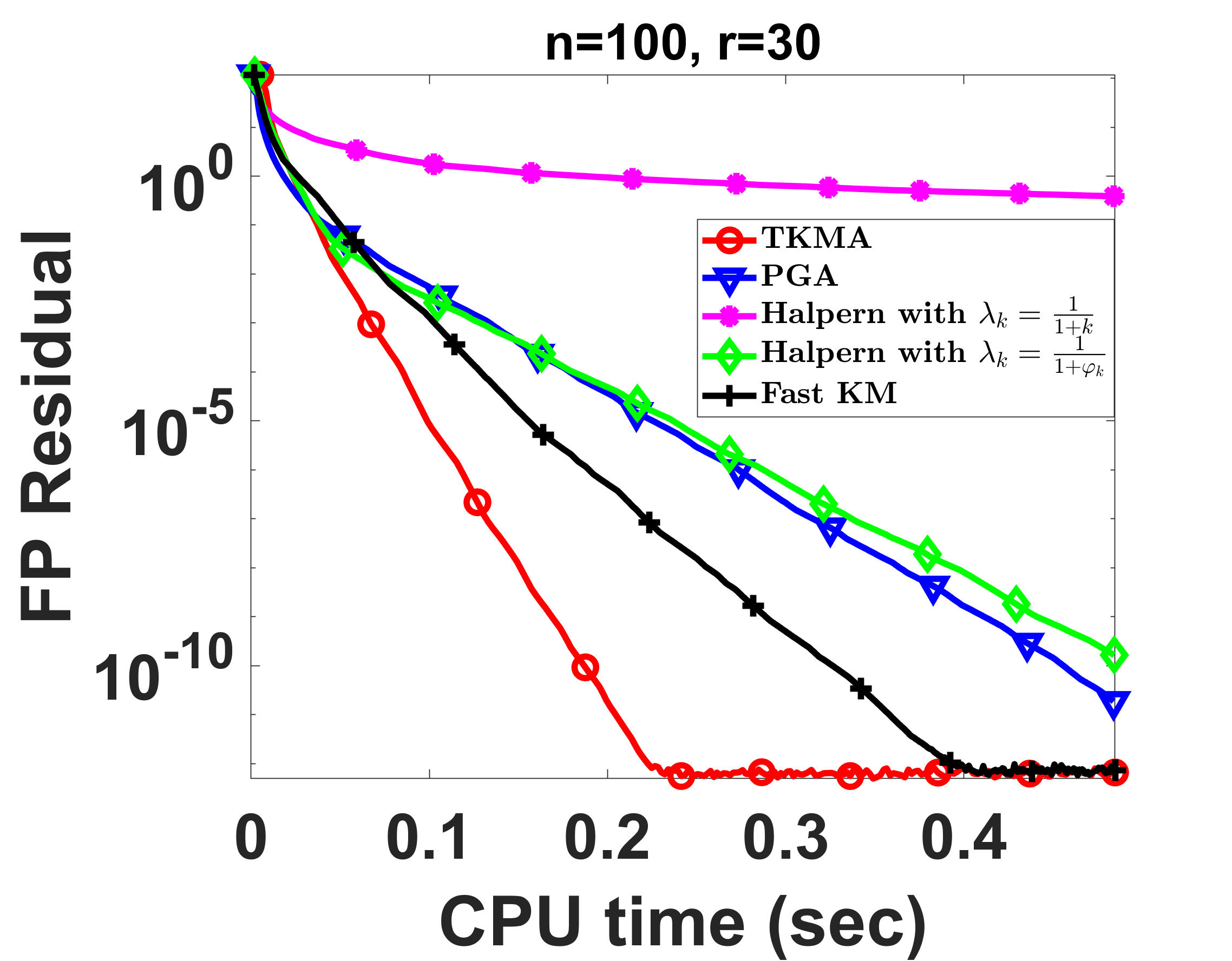}&
\includegraphics[width=0.475\textwidth, height=0.23\textheight]{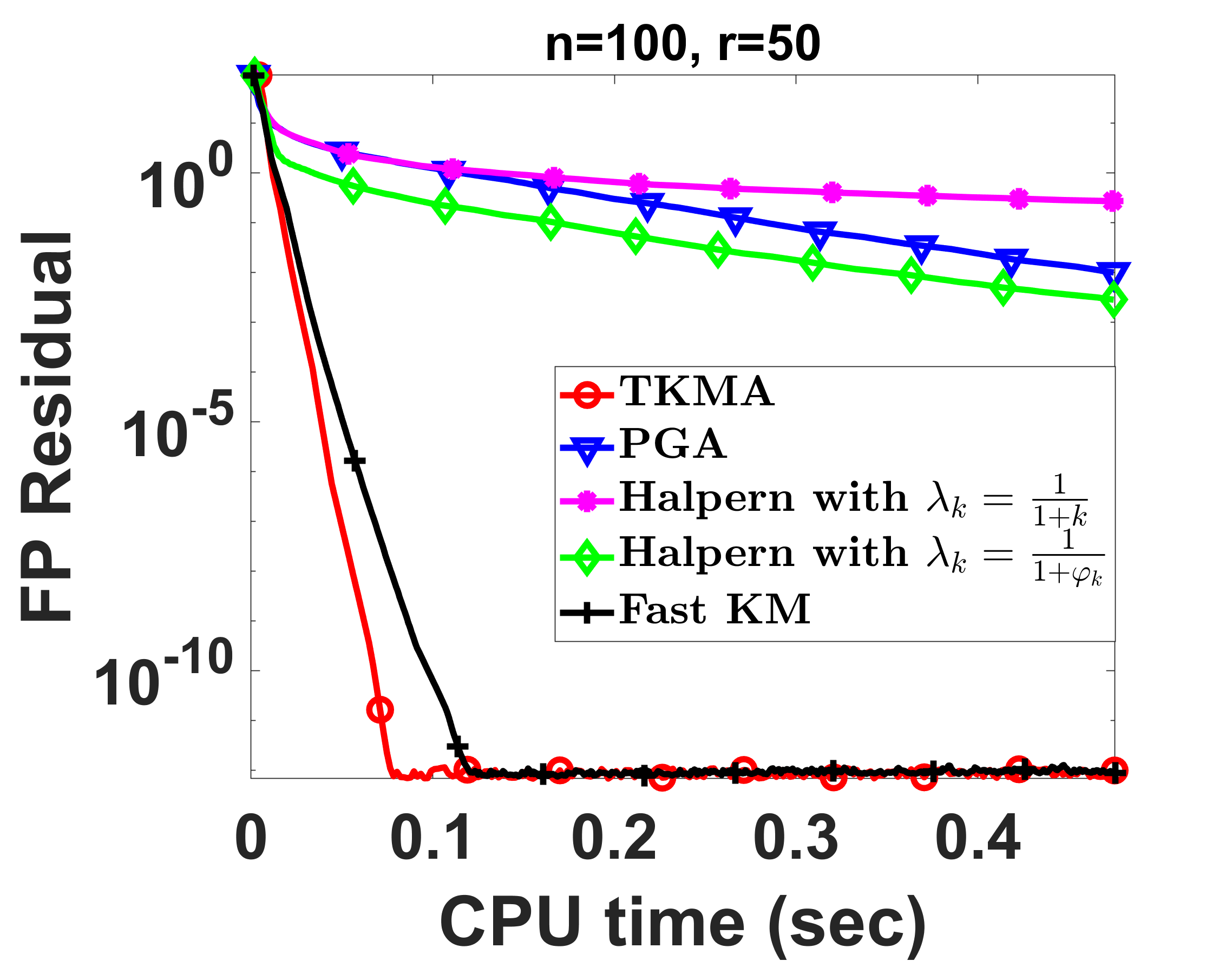}
\end{tabular}
\vspace{-0.7em}\caption{FP residual against CPU time for TKMA, PGA, Halpern algorithm with $\lambda_k=\frac{1}{k+1}$, Halpern algorithm with $\lambda_k=\frac{1}{\varphi_k+1}$,
and Fast KM algorithm with ranks $r=30$ (left) and $r=50$ (right).}
\label{residualr3050}
\end{figure}

\begin{figure}[htbp]
\centering
\begin{tabular}{cc}
\hspace{-1em}
\includegraphics[width=0.47\textwidth]{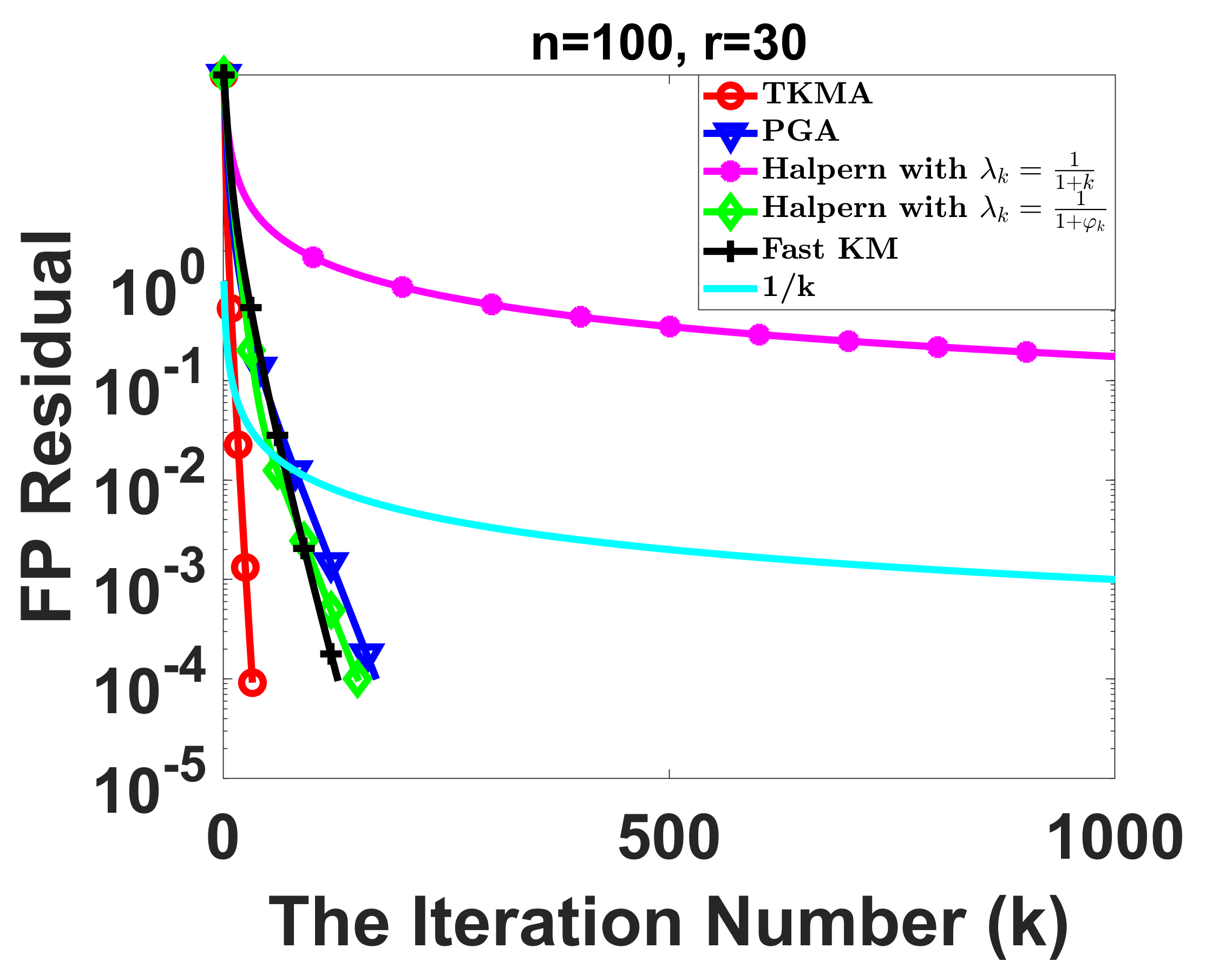}&
\includegraphics[width=0.47\textwidth]{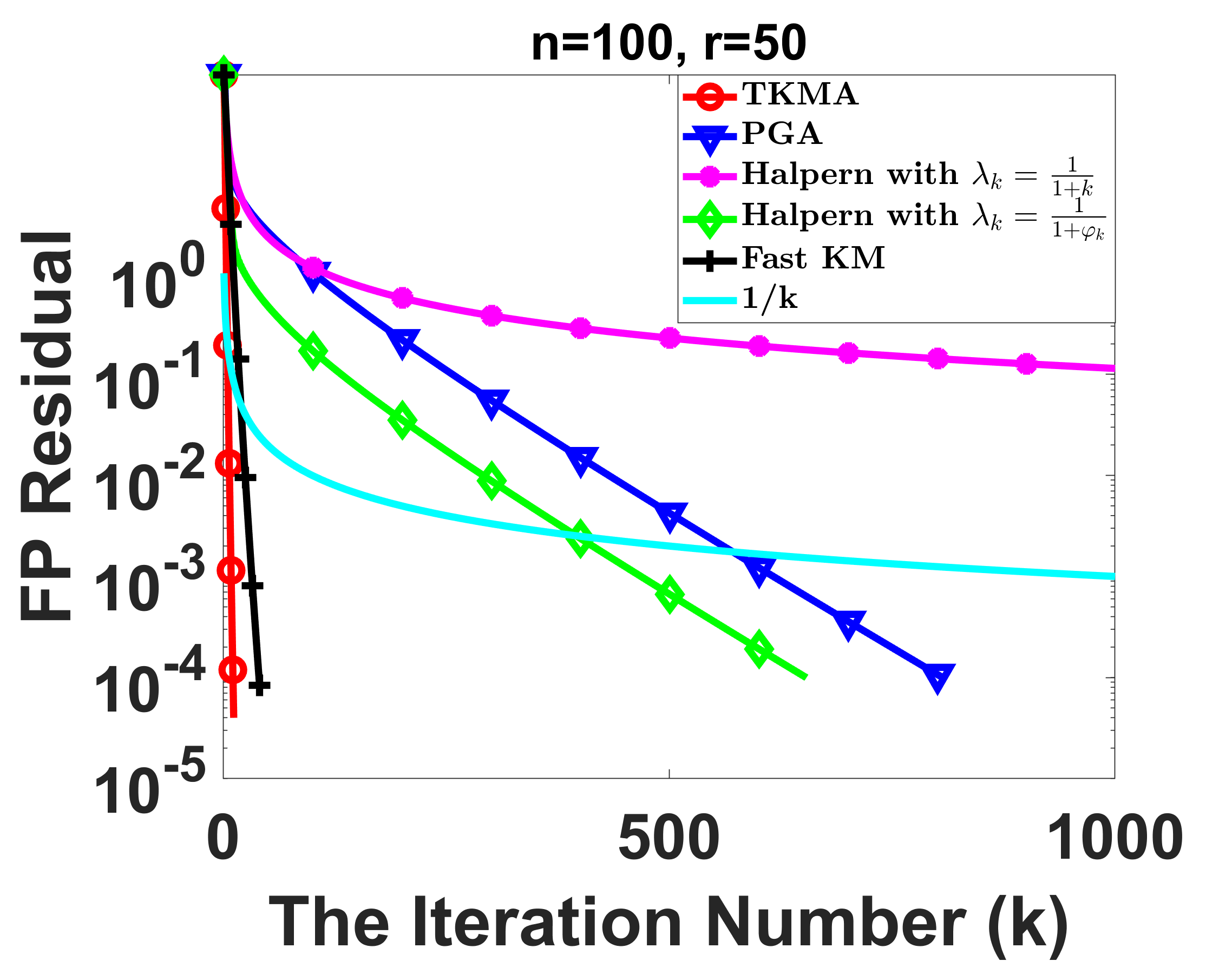}
\end{tabular}
\vspace{-0.5em}\caption{FP residual versus the iteration number for TKMA, PGA, Halpern algorithm with $\lambda_k=\frac{1}{k+1}$, Halpern algorithm with $\lambda_k=\frac{1}{\varphi_k+1}$,
and Fast KM algorithm with ranks $r=30$ (left) and $r=50$ (right).}
\label{Residualmatrix}
\end{figure}

\section{Conclusion}
In this work, we introduce a Two-step Krasnosel'ski\u{\i}-Mann Algorithm with adaptive momentum (TKMA), offering a new perspective for enhancing fixed-point methods in convex optimization. This algorithm extends the classical KM framework by incorporating novel adaptive momentum  and a two-step iterative structure. Theoretical analysis establishes both the weak convergence of the TKMA and an $o(1/\sqrt{k})$ convergence rate in terms of the distance between consecutive iterates under certain assumptions in a real Hilbert space. Numerical results on image denoising and low-rank matrix completion validate the practical advantage of the TKMA over several state-of-the-art methods, demonstrating its potential as an effective tool for large-scale imaging and data optimization problems.

\newpage
\bibliographystyle{siamplain}
\bibliography{bibfile}



\end{sloppypar}
\end{document}

%% file: siims_shared.tex
\usepackage{lipsum}
\usepackage{amsfonts}
\usepackage{graphicx}
\usepackage{epstopdf}
\usepackage{algorithmic}
\usepackage{amsmath}
\usepackage{amssymb}
\usepackage{algorithm}
\usepackage{mathrsfs}
\usepackage{bm}
\usepackage{float}
\usepackage{multirow}
\usepackage{color}
\usepackage{booktabs}
\usepackage{graphicx}
\usepackage{subfig}
\usepackage{diagbox}
\usepackage{url}
\usepackage{verbatim}
\usepackage{caption}

\ifpdf
  \DeclareGraphicsExtensions{.eps,.pdf,.png,.jpg}
\else
  \DeclareGraphicsExtensions{.eps}
\fi

\newsiamremark{remark}{Remark}
\newsiamremark{hypothesis}{Hypothesis}
\crefname{hypothesis}{Hypothesis}{Hypotheses}
\newsiamthm{claim}{Claim}
\newsiamthm{example}{Example}
\newsiamthm{fact}{Fact}

\headers{A Two-step KM Algorithm with Adaptive Momentum}{Y. He, J. Li, Y. Lin, and D. Han}

\title{A Two-step Krasnosel'ski\u{\i}-Mann Algorithm with Adaptive Momentum and Its Applications to Image Denoising and Matrix Completion\thanks{The first edition published on arXiv was on October 28th, 2025.}
\funding{This work was supported in part by the National Natural Science Foundation of China under Grants 12401120 and 12131004, and by Guangdong Basic and Applied Basic Research Foundation under Grant 2025A1515011544.}}

\author{Yongxin He$^\dagger$\thanks{$^\dagger$Department of Mathematics, Jinan University, Guangzhou, China (\email{wheyongxin@163.com, cyciliali@163.com, linyizun@jnu.edu.cn}).}
\and Jingyuan Li$^\dagger$
\and Yizun Lin$^{\dagger*}$
\and Deren Han$^\ddagger$\thanks{$^\ddagger$School of Mathematical Sciences, Beihang University, Beijing, China (\email{handr@buaa.edu.cn}); School of Mathematics and Computational Science, Xiangtan University, Xiangtan, China.}
}

\makeatletter
\def\thanks#1{\protected@xdef\@thanks{\@thanks\protect\footnotetext[0]{#1}}}
\makeatother

\usepackage{amsopn}
\DeclareMathOperator*{\argmin}{argmin}

\def \bbR {\mathbb R}
\def \bbN {\mathbb N}

\def \mA {\mathcal{A}}

\def \mH {\mathcal{H}}
\def \mI {\mathcal{I}}
\def \mN {\mathcal{N}}

\def \bmA {{\bm A}}
\def \bmB {{\bm B}}
\def \bmD {{\bm D}}
\def \bmX {{\bm X}}
\def \bmY {{\bm Y}}
\def \bmZ {{\bm Z}}
\def \bmU {{\bm U}}
\def \bmV {{\bm V}}

\def \bmI {{\bm I}}

\def \bmh {{\bm h}}

\def \bmu {{\bm u}}
\def \bmv {{\bm v}}
\def \bms {{\bm s}}
\def \bmw {{\bm w}}
\def \bmx {{\bm x}}
\def \bmy {{\bm y}}
\def \bmz {{\bm z}}
\def \bmx {{\bm x}}

\def \bbRd {\bbR^{d}}

\def \bbRnn {\bbR^{n\times n}}

\def \prox {\mathrm{prox}}

\def \Fix {\text{Fix}}

\def \hx {\hat{\bmx}}


%% file: TKMA.bbl
\begin{thebibliography}{10}

\bibitem{apostol1958mathematical}
{\sc T.~M. Apostol and C.~Ablow}, {\em Mathematical {Analysis}},
  Addison-Wesley, California, 1974.

\bibitem{bauschke2017convex}
{\sc H.~H. Bauschke and P.~L. Combettes}, {\em Convex Analysis and Monotone
  Operator Theory in {Hilbert} Space}, Springer, New York, 2nd~ed., 2017.

\bibitem{beck2017first}
{\sc A.~Beck}, {\em First-Order Methods in Optimization}, SIAM, Philadelphia,
  {PA}, 2017.

\bibitem{beck2009fast}
{\sc A.~Beck and M.~Teboulle}, {\em A fast iterative shrinkage-thresholding
  algorithm for linear inverse problems}, SIAM J. Imaging Sci., 2 (2009),
  pp.~183--202.

\bibitem{bot2023fast}
{\sc R.~I. Bot and D.-K. Nguyen}, {\em Fast {Krasnosel’ski\u{\i}-Mann}
  algorithm with a convergence rate of the fixed point iteration of o(1/k)},
  SIAM J. Numer. Anal., 61 (2023), pp.~2813--2843.

\bibitem{chen2015convergence}
{\sc B.~Chen, J.~Wang, H.~Zhao, N.~Zheng, and J.~C. Principe}, {\em Convergence
  of a fixed-point algorithm under maximum correntropy criterion}, IEEE Signal
  Process Lett., 22 (2015), pp.~1723--1727.

\bibitem{chen2013primal}
{\sc P.~Chen, J.~Huang, and X.~Zhang}, {\em A primal--dual fixed point
  algorithm for convex separable minimization with applications to image
  restoration}, Inverse Probl., 29 (2013), p.~025011.

\bibitem{davis2017convergence}
{\sc D.~Davis and W.~Yin}, {\em Convergence rate analysis of several splitting
  schemes}, in Splitting methods in communication, imaging, science, and
  engineering, Springer, 2017, pp.~115--163.

\bibitem{figueiredo2003algorithm}
{\sc M.~A. Figueiredo and R.~D. Nowak}, {\em An {EM} algorithm for
  wavelet-based image restoration}, IEEE Trans. Image Process., 12 (2003),
  pp.~906--916.

\bibitem{figueiredo2008gradient}
{\sc M.~A. Figueiredo, R.~D. Nowak, and S.~J. Wright}, {\em Gradient projection
  for sparse reconstruction: Application to compressed sensing and other
  inverse problems}, IEEE J. Sel. Top. Signal Process., 1 (2008), pp.~586--597.

\bibitem{goebel1984uniform}
{\sc K.~Goebel and S.~Reich}, {\em Uniform convexity, hyperbolic geometry, and
  non-expansive mappings}, Marcel Dekker,  (1984).

\bibitem{groetsch1972note}
{\sc C.~Groetsch}, {\em A note on segmenting {Mann} iterates}, J. Math. Anal.
  Appl., 40 (1972), pp.~369--372.

\bibitem{halpern1967fixed}
{\sc B.~Halpern}, {\em Fixed points of nonexpanding maps}, Bull. Amer. Math.
  Soc., 73 (1967), pp.~957--961.

\bibitem{he2024convergence}
{\sc S.~He, H.-K. Xu, Q.-L. Dong, and N.~Mei}, {\em Convergence analysis of the
  {Halpern} iteration with adaptive anchoring parameters}, Math. Comput., 93
  (2024), pp.~327--345.

\bibitem{kanzow2017generalized}
{\sc C.~Kanzow and Y.~Shehu}, {\em Generalized
  {Krasnosel’ski\u{\i}--Mann}-type iterations for nonexpansive mappings in
  {Hilbert} spaces}, Comput. Optim. Appl., 67 (2017), pp.~595--620.

\bibitem{krasnosel1955two}
{\sc M.~A. Krasnosel’ski\u{\i}}, {\em Two remarks on the method of successive
  approximations}, Usp. Mat. Nauk, 10 (1955), pp.~123--127.

\bibitem{krol2012preconditioned}
{\sc A.~Krol, S.~Li, L.~Shen, and Y.~Xu}, {\em Preconditioned alternating
  projection algorithms for maximum a posteriori {ECT} reconstruction}, Inverse
  Probl., 28 (2012), p.~115005.

\bibitem{li2015multi}
{\sc Q.~Li, L.~Shen, Y.~Xu, and N.~Zhang}, {\em Multi-step fixed-point
  proximity algorithms for solving a class of optimization problems arising
  from image processing}, Adv. Comput. Math., 41 (2015), pp.~387--422.

\bibitem{li2018fixed}
{\sc Z.~Li, G.~Song, and Y.~Xu}, {\em A fixed-point proximity approach to
  solving the support vector regression with the group lasso regularization},
  Int. J. Numer. Anal. Model, 15 (2018), pp.~154--169.

\bibitem{li2019two}
{\sc Z.~Li, G.~Song, and Y.~Xu}, {\em A two-step fixed-point proximity
  algorithm for a class of non-differentiable optimization models in machine
  learning}, J. Sci. Comput., 81 (2019), pp.~923--940.

\bibitem{liang2016convergence}
{\sc J.~Liang, J.~Fadili, and G.~Peyr{\'e}}, {\em Convergence rates with
  inexact non-expansive operators}, Math. Program., 159 (2016), pp.~403--434.

\bibitem{lieder2021convergence}
{\sc F.~Lieder}, {\em On the convergence rate of the {Halpern}-iteration},
  Optim. Lett., 15 (2021), pp.~405--418.

\bibitem{lin2025accelerated}
{\sc Y.~Lin, Y.~He, C.~R. Schmidtlein, and D.~Han}, {\em An accelerated
  preconditioned proximal gradient algorithm with a generalized {Nesterov}
  momentum for {PET} image reconstruction}, Inverse Probl., 41 (2025),
  p.~045002.

\bibitem{lin2019krasnoselskii}
{\sc Y.~Lin, C.~R. Schmidtlein, Q.~Li, S.~Li, and Y.~Xu}, {\em A
  {Krasnosel’ski\u{\i}-Mann} algorithm with an improved {EM} preconditioner
  for {PET} image reconstruction}, IEEE Trans. Med. Imaging, 38 (2019),
  pp.~2114--2126.

\bibitem{lin2022convergence}
{\sc Y.~Lin and Y.~Xu}, {\em Convergence rate analysis for fixed-point
  iterations of generalized averaged nonexpansive operators}, J. Fix. Point
  Theroy Appl., 24 (2022).

\bibitem{lu2016multiplicative}
{\sc J.~Lu, L.~Shen, C.~Xu, and Y.~Xu}, {\em Multiplicative noise removal in
  imaging: An exp-model and its fixed-point proximity algorithm}, Appl. Comput.
  Harmon. Anal., 41 (2016), pp.~518--539.

\bibitem{micchelli2011proximity}
{\sc C.~A. Micchelli, L.~Shen, and Y.~Xu}, {\em Proximity algorithms for image
  models: denoising}, Inverse Probl., 27 (2011), p.~045009.

\bibitem{micchelli2013proximity}
{\sc C.~A. Micchelli, L.~Shen, Y.~Xu, and X.~Zeng}, {\em Proximity algorithms
  for the {L1/TV} image denoising model}, Adv. Comput. Math., 38 (2013),
  pp.~401--426.

\bibitem{opial1967weak}
{\sc Z.~Opial}, {\em Weak convergence of the sequence of successive
  approximations for nonexpansive mappings}, Bull. Amer. Math. Soc., 73 (1967),
  pp.~591--597.

\bibitem{park2022exact}
{\sc J.~Park and E.~K. Ryu}, {\em Exact optimal accelerated complexity for
  fixed-point iterations}, in International Conference on Machine Learning,
  PMLR, 2022, pp.~17420--17457.

\bibitem{polson2015proximal}
{\sc N.~G. Polson, J.~G. Scott, and B.~T. Willard}, {\em Proximal algorithms in
  statistics and machine learning}, Stat. Sci.,  (2015), pp.~559--581.

\bibitem{qi2021convergence}
{\sc H.~Qi and H.-K. Xu}, {\em Convergence of {Halpern’s} iteration method
  with applications in optimization}, Numer. Funct. Anal. Optim., 42 (2021),
  pp.~1839--1854.

\bibitem{ross2017relaxed}
{\sc C.~Ross~Schmidtlein, Y.~Lin, S.~Li, A.~Krol, B.~J. Beattie, J.~L. Humm,
  and Y.~Xu}, {\em Relaxed ordered subset preconditioned alternating projection
  algorithm for {PET} reconstruction with automated penalty weight selection},
  Med. Phys., 44 (2017), pp.~4083--4097.

\bibitem{schaefer1957uber}
{\sc H.~Schaefer}, {\em Uber die methode sukzessive. approximationen}, Iber.
  Deutch. Math. Verein., 59 (1957), pp.~131--140.

\bibitem{shen2016wavelet}
{\sc L.~Shen, Y.~Xu, and X.~Zeng}, {\em Wavelet inpainting with the $\ell_0$
  sparse regularization}, Appl. Comput. Harmon. Anal., 41 (2016), pp.~26--53.

\bibitem{yao2008weak}
{\sc Y.~Yao and Y.-C. Liou}, {\em Weak and strong convergence of
  {Krasnosel’ski\u{\i}--Mann} iteration for hierarchical fixed point
  problems}, Inverse Probl., 24 (2008), p.~015015.

\bibitem{yoon2021accelerated}
{\sc T.~Yoon and E.~K. Ryu}, {\em Accelerated algorithms for smooth
  convex-concave minimax problems with {$O(1/k^2)$} rate on squared gradient
  norm}, in International Conference on Machine Learning, PMLR, 2021,
  pp.~12098--12109.

\bibitem{zhang2025hppp}
{\sc S.~Zhang, H.~Zhang, and H.~Wang}, {\em {HPPP}: Halpern-type preconditioned
  proximal point algorithms and applications to image restoration}, SIAM J.
  Imaging Sci., 18 (2025), pp.~1493--1521.

\bibitem{zheng2019sparsity}
{\sc W.~Zheng, S.~Li, A.~Krol, C.~R. Schmidtlein, X.~Zeng, and Y.~Xu}, {\em
  Sparsity promoting regularization for effective noise suppression in {SPECT}
  image reconstruction}, Inverse Probl., 35 (2019), p.~115011.

\bibitem{zhu2015fast}
{\sc Y.~Zhu, J.~Wu, and G.~Yu}, {\em A fast proximal point algorithm for
  $\ell_1$-minimization problem in compressed sensing}, Appl. Math. Comput.,
  270 (2015), pp.~777--784.

\end{thebibliography}
